\documentclass[10pt]{amsart}
\usepackage{amsmath}
\usepackage{amsfonts}
\usepackage{amssymb}
\usepackage{amsthm}
\usepackage{mathrsfs}
\usepackage{graphicx}
\usepackage{hyperref}

\usepackage{soul}

\usepackage{color}

\def \de {\partial}

\def \N {\mathbb{N}}

\def \phi {\varphi}

\def \R {\mathbb{R}}

\def \Lop {\mathcal{L}}

\def \div {\operatorname{div}}
\def \tr {\operatorname{tr}}

\def \erren {{\mathbb{R}^N}}
\def \erreu {{\mathbb{R}^{N+1}}}
\def \erre {\mathbb{R}}
\def \cappa {\mathcal{K}}
\def \eps {\varepsilon}
\def \ucap {\hat u}
\def \Par {\mathcal{P}}

\newtheorem{theorem}{Theorem}[section]
\newtheorem{lemma}[theorem]{Lemma}
\newtheorem{proposition}[theorem]{Proposition}
\newtheorem{corollary}[theorem]{Corollary}

\newtheorem{remark}[theorem]{Remark}

\theoremstyle{definition}

\numberwithin{equation}{section}

\begin{document}

\title[One-side Liouville Theorem for hypoelliptic Ornstein--Uhlenbeck operators]{One-side Liouville Theorem \\  for hypoelliptic Ornstein--Uhlenbeck operators \\ having drifts with imaginary spectrum}

\author{Alessia E. Kogoj}
\address{Dipartimento di Scienze Pure e Applicate (DiSPeA)\\ 
				 Universit\`{a} degli Studi di Urbino Carlo Bo\\
				 Piazza della Repubblica 13, 61029 Urbino (PU), Italy.}
\email{alessia.kogoj@uniurb.it}

\author{Ermanno Lanconelli}
\address{Dipartimento di Matematica\\ 
				 Alma Mater Studiorum Universit\`{a} di Bologna\\
				 Piazza di Porta San Donato 5, 40126 Bologna, Italy.}
\email{ermanno.lanconelli@unibo.it}

\author{Giulio Tralli}
\address{Dipartimento di Matematica e Informatica\\ 
				 Universit\`{a} degli Studi di Ferrara\\
				 Via Machiavelli 30, 44121 Ferrara, Italy.}
\email{giulio.tralli@unife.it}

\date{}

\begin{abstract}
We prove the Liouville theorem for \emph{non-negative} solutions to (possibly degenerate) Ornstein--Uhlenbeck equations whose linear drift has imaginary spectrum. This provides an answer to a question raised by Priola and Zabczyk since the proof of their Theorem characterizing the Ornstein-Uhlenbeck operators having the Liouville property for \emph{bounded} solutions. Our approach is based on a Liouville property at ``$t=-\infty$" for the solutions to the relevant Kolmogorov equation which, in turn, derives from a new parabolic Harnack-type inequality for its non-negative ancient solutions.
\end{abstract}
\maketitle
\section{Introduction and main results}

Let $A$ and $B$ be real $N\times N$ matrices, and let $A$ be symmetric non-negative definite. The linear second order PDO in $\erren$
\begin{equation}\label{Lop}
\Lop:=\operatorname{tr}\left(AD^2\right)+\left\langle Bx,D\right\rangle,
\end{equation}
is usually called the Ornstein-Uhlenbeck  operator associated to $(A,B).$ Here and in what follows, $D, D^2, \langle \cdot,\cdot \rangle,$ and $x$ denote, respectively, the gradient and the Hessian operators, the inner product, and the generic point of $\erren$. 

Letting, 
\begin{equation}\label{1.2} E(t):=e^{-tB},\end{equation}
one defines the $N\times N$ real symmetric matrix 
\begin{equation}\label{Kmatrix}
C(t):=\int_0^t E(s)AE(s)^T ds. 
\end{equation}
Throughout the paper we assume \begin{equation}\label{hyphorm}
C(t)>0 \mbox{ for every } t>0.
\end{equation}
As it is quite well known, condition \eqref{hyphorm} is equivalent to the {\it hypoellipticity} of $\Lop$, i.e., to the smoothness of the distributional solutions of $\Lop u=f$ whenever $f$ is smooth (see e.g. \cite{LP}, and the references therein). In particular, \eqref{hyphorm} is always satisfied in case $A$ is strictly positive definite, and it can rephrased in terms of an algebraic condition involving $Ker(A)$ and $B^T$ (see \cite[Appendix]{LP} and Section \ref{secMatrix} below).

One says that $\Lop$ has the 
\begin{center} {\it $L^\infty$-Liouville property}
\end{center}
if every bounded (smooth) solution to 
\begin{equation}\label{1.5} \Lop u = 0 \mbox{ in } \erren
\end{equation}
is constant. 

One says that $\Lop$ has the 
\begin{center} {\it one-side Liouville property}
\end{center}
if the constant functions are the unique bounded below (smooth) solutions to \eqref{1.5}. 

Priola and Zabczyk - in 2004 - proved the following noteworthy and elegant theorem in \cite[Theorem 3.1]{PZ}:\\
{\it Let $\Lop$ be the hypoelliptic Ornstein--Uhlenbeck operator in \eqref{Lop}. Then $\Lop$ has the $L^\infty$-Liouville property if and only if 

\begin{equation}\label{1.6} \mathrm{Re}(\lambda)\le 0 \mbox { for every } \lambda\in\sigma(B).\end{equation}
}

Here, and in what follows, $\sigma(B)$ denotes the spectrum of $B$, i.e., the set of eigenvalues of $B$.

The Priola--Zabczyk Theorem naturally raises the following question:

\begin{itemize}
\item[(Q)] Does condition \eqref{1.6} also guarantee the stronger one-side Liouville property for $\Lop$? 

\end{itemize}
This question was raised by Priola and Zabczyk themselves in \cite[Section 6, Open questions]{PZquaderni} and so far, to the best of our knowledge, it has been answered in the specific cases (in which $\Lop$ is always assumed to be hypoelliptic) listed here:
\begin{itemize}
\item[(I)] The operator $\Lop$ is homogeneous of degree two with respect to a group of dilations in $\erren$ (see \cite[Corollary 8.3]{mediterranean}). We stress that this condition implies $\sigma(B)=\{0\}$.
\item[(II)] The matrix $A$ is strictly positive definite, the matrix $B$ is diagonalizable over the complex field and $\sigma(B)\subseteq i \mathbb{R}$ (see \cite[Theorem 1.1]{KLP}).
\item[(III)] The real Jordan representation of $B$ is 
$$\begin{pmatrix}
B_0 & 0\\
0 &B_1 
\end{pmatrix}$$
where $\sigma(B_0)\subseteq \{ \lambda\in \mathbb{C}\ |\ \mathrm{Re}(\lambda)<0\}$ and $B_1$ is at most of dimension $2$
and of the form 
$$ B_1=\begin{pmatrix} 0\end{pmatrix}\qquad
 \mbox{ or }\qquad B_1=\begin{pmatrix}
0 & -\alpha\\
\alpha &0 \end{pmatrix}\,\,\,\mbox{with $\alpha\in\mathbb{R}$} $$
(see \cite[Theorem 6.1]{KLP} and the references therein).
\end{itemize}
We also quote a recent result by Priola in \cite{P} where it is treated the case of solutions having a controlled growth at infinity under two sets of assumptions.

The main purpose of this paper is to give the following answer to the question (Q). 
\begin{theorem}\label{main} Let $\Lop$ be the hypoelliptic Ornstein--Uhlenbeck operator in \eqref{Lop} and assume 
\begin{equation}\label{hypreal0}
\sigma(B)\subseteq  i \mathbb{R}.\end{equation}
Then $\Lop$ has the one-side Liouville property.
\end{theorem}

This theorem obviously contains the Liouville theorems of the cases (I) and (II) while only complements the one of (III). 

Theorem \ref{main} is a simple corollary of the following Liouville theorem at ``$t=-\infty$" for the evolution counterpart of $\Lop$, i.e., for the Kolmogorov-type operator in $\erreu$ 
\begin{equation}\label{Lopt}  \mathcal{K}:=\Lop-\de_t.\end{equation}
\begin{theorem}\label{ancient} Let $\Lop$ be the hypoelliptic Ornstein--Uhlenbeck operator in \eqref{Lop} and let $\mathcal{K}$ be the corresponding Kolmogorov operator in \eqref{Lopt}. Assume 
$$\sigma(B)\subseteq  i \mathbb{R}.$$
If $u$ is a smooth solution to $$\mathcal{K}u =0 \mbox{ in } \erreu$$ such that
$$\inf_{\erreu} u > -\infty,$$
then, 
$$\lim_{t \to -\infty} u(x,t) = \inf_{\erreu} u$$
for every $x\in\erren$.
\end{theorem}

We want to immediately show that Theorem \ref{ancient} implies Theorem \ref{main}. 

\begin{proof}[Proof of Theorem \ref{main}]
Let $v:\erren\longrightarrow\erre$ be a smooth and bounded below solution to 
$$\Lop v=0 \mbox{ in } \erren.$$ Then, the function
$$u:\erreu\longrightarrow \erre,\quad u(x,t)=v(x)$$
is smooth and solves 
$$\cappa u = 0  \mbox{ in } \erreu.$$ 
Moreover, 
\begin{equation}\label{inf} \inf_\erreu u =\inf_\erren v >-\infty.
\end{equation}
Then, by Theorem \ref{ancient},
$$ \lim_{t \to -\infty} u(x,t)=  \inf_\erreu u \mbox{ for every } x\in\erren.$$
On the other hand,
$$ \lim_{t \to -\infty} u(x,t)=  \lim_{t \to -\infty} v(x)=v(x).$$
Hence, also keeping in  mind \eqref{inf},
$$v(x)=\inf_\erren v \mbox{ for every } x\in\erren,$$
that is $v$ is constant and Theorem \ref{main} is proved.
\end{proof}
The remaining part of this paper is entirely devoted to the proof of Theorem \ref{ancient}. Our key idea for this proof, inspired by a procedure first used in \cite{KLP}, consists in deducing the Liouville theorem for $\cappa$ at ``$t=-\infty$" from a kind of global 
Harnack inequality for the positive entire solution to $\cappa u=0$ (cfr. \eqref{eccoH}). In turn, our proof of this Harnack inequality uses, as in \cite{KLP}, suitable Mean Value formulas for $\cappa$-harmonic functions. The main technical difference, in comparison with previous works, relies on quantitative estimates for the relevant geometric objects which we show to be uniform for large (negative) times under the sole assumption \eqref{hypreal0}.

The paper is organized as follows. In Section \ref{secPre} we fix the needed notations. In Section \ref{secMatrix} we establish quantitative estimates with respect to the time parameter $t$ for the quadratic forms generated by $C(t)$ and $E(-t)C(t)E(-t)^T$: in particular we prove a global doubling property for the determinant function in Corollary \ref{corDoub} as well as crucial self-similarity properties along time in Theorem \ref{propmatrix2}. These results, which we believe are of independent interest, are exploited in Sections  \ref{sec4}-\ref{sec5}-\ref{sec6} where we address, respectively, the following properties: the large-time behaviour of the \emph{intrinsic paraboloids} $\mathcal{P}_{z_0}$, an engulfing property for the sets $\Omega^{(p)}_r(z_0)$ where the mean value formulas are defined, and a pointwise comparison for the kernels of such mean value formulas. These three results come into play in Section \ref{sec7} where we conclude the proof of Theorem \ref{ancient} by showing the Harnack-type inequality.

\section{Preliminaries and further notations}\label{secPre}

We denote the generic point $z\in \erreu$ as follows
$$z=(x,t), x\in \erren, t\in\erre.$$
Let $C(t)$ be the Kalman matrix in \eqref{Kmatrix}, and $E(t)$ the exponential-type matrix in \eqref{1.2}. Since $E(-t)E(s)=E(s-t)$, the following holds true for any positive $t$
\begin{equation}\label{deftKt}
E(-t)C(t)E(-t)^T=\int_0^t E(s-t)AE(s-t)^T ds=\int_0^t E(-\sigma)AE(-\sigma)^T d\sigma.\end{equation}
It is clear from their definitions that both the matrices $C(t)$ and $E(-t)C(t)E(-t)^T$ are monotone increasing with respect to $t$ in the sense of quadratic forms, and one can switch from one to the other with the change $B\mapsto -B$: both these matrices will play a pervasive role in this work. Since $B$ is a real matrix and $\sigma (B)\subseteq i \erre$ by \eqref{hypreal0}, if $\lambda \in \sigma(B)$ then $-\lambda \in\sigma(B)$. As a consequence
$$\div (Bx)=\operatorname{tr}(B)=0.$$
In particular, for any $t>0$ one has $\operatorname{det}\left(E(-t)\right)=e^{t\operatorname{tr}(B)}=1$ and we can define
\begin{equation}\label{defd}
D(t):=\operatorname{det}\left(C(t)\right)=\operatorname{det}\left(E(-t)C(t)E(-t)^T\right).
\end{equation}
For $t>0$ and $p\in \mathbb{N}$, we also fix the notation 
\begin{equation}\label{defV}
V_p(t):=\sqrt{(4\pi)^{N+p}\, t^p\, D(t)}.
\end{equation}

Under our assumption \eqref{hyphorm}, the PDO in $\erreu$ $$\cappa=\Lop -\de_t$$
has a fundamental solution 
$$\Gamma: \{ (z_0,z) \in \erreu\times \erreu\ | \ z_0\neq z\}\longrightarrow \erre$$
given by
$$\Gamma(z_0;z)= \Gamma((x_0,t_0);(x,t))=0$$
if $t_0\le t$, and 
\begin{eqnarray*}\Gamma(z_0;z)&=& \Gamma((x_0,t_0);(x,t))\\&=&\frac{(4\pi)^{-\frac{N}{2}}}{\sqrt{D(t_0-t)}} \exp{ \left(  -\frac{\langle C^{-1}(t_0-t) \left(x_0-E(t_0-t)x\right), \left(x_0-E(t_0-t)x\right) \rangle}{4}   \right)}\end{eqnarray*}
if $t<t_0$ (see e.g. \cite{LP}). The $\cappa$-harmonic functions, i.e., the solutions to $\cappa u=0$, satisfy Mean Value formulas on suitable superlevel sets of $\Gamma$. To show these formulas, we need some more notations. 

Fix $(x_0,t_0)\in\erreu$. For every $p\in\mathbb{N}$ and $r>0$, define 
\begin{equation}\label{2.2}
\Omega^{(p)}_r(z_0):=\left\{z=(x,t)\in\R^{N+1}\,:\, (4\pi (t_0-t))^{-\frac{p}{2}} \Gamma(z_0;z) >\frac{1}{r}\right\},
\end{equation}
and
\begin{equation}\label{2.3}
W_r^{(p)}(z_0;z):=\omega_p R^p_r(z_0;z)\left(W(z_0;z)+ \frac{p}{4(p+2)}\left(\frac{R_r(z_0;z)}{t_0-t}\right)^2\right).
\end{equation}
Here, $\omega_p$ is the volume of the unit ball of $\erre^p$, and we adopt the notations
\begin{equation}\label{2.4}
R_r(z_0;z):=\sqrt{4(t_0-t)\log{\left(r (4\pi (t_0-t))^{-\frac{p}{2}} \Gamma(z_0;z) \right)}}
\end{equation}
and
\begin{align}\label{2.6}
&W(z_0;z):=\\
&=\frac{1}{4}\left\langle \left(E(t_0-t) A E(t_0-t)^T\right) C^{-1}(t_0-t) \left(x_0-E(t_0-t)x\right), C^{-1}(t_0-t) \left(x_0-E(t_0-t)x\right) \right\rangle \notag\\ 
&=\frac{1}{4}\left\langle A \left(E(t-t_0)C(t_0-t)E(t-t_0)^T\right)^{-1}\left(x-E(t-t_0)x_0\right),\right.\notag\\  
&\hspace{5.6cm}\left.\left(E(t-t_0)C(t_0-t)E(t-t_0)^T\right)^{-1}\left(x-E(t-t_0)x_0\right)\right\rangle.\notag
\end{align}

Then, every solution to $\cappa u=0$ in $\erreu$ satisfies the Mean Value formula 
\begin{equation} \label{mvf} u(z_0)= \frac{1}{r}\int_{\Omega^{(p)}_r(z_0)} u(z) W_r^{(p)}(z_0;z) dz
\end{equation}
for every $z_0\in\erreu$, for every $r>0$ and for every $p\in \mathbb{N}$. A complete and detailed proof of this formula can be found in the paper \cite{cup_lan_media}.
\section{Matrix algebra and doubling property}\label{secMatrix}

In this section we study the behaviour with respect to the time parameter $t$ of the quadratic forms $\left\langle C(t)\xi,\xi\right\rangle$ and $\left\langle E(-t)C(t)E(-t)^T\xi,\xi\right\rangle$. The assumption \eqref{hyphorm} readily implies the existence of a positive smallest eigenvalue which depends on $t$. We aim at establishing estimates which are uniform in $t$. By assuming only \eqref{hyphorm}, the small time (uniform) behaviour is already known thanks to the analysis in \cite{LP}. In particular, as a by-product of the results in \cite{LP} we know that there exists $n_0\in\N\cup\{0\}$ such that the following holds: for any $T_0>0$ there exists a constant $K(T_0)\geq 1$ such that
\begin{equation}\label{fixedhorizon}
\frac{t^{2n_0+1}}{K(T_0)} |\xi|^2\leq \left\langle E(-t)C(t)E(-t)^T\xi,\xi \right\rangle \leq K(T_0) t |\xi|^2 \quad\forall \xi\in\R^N\mbox{ and for every }t\in (0,T_0].\end{equation}
We invite the reader to look at \cite[Lemma 3.3. and Proposition 2.3]{LP} and \cite[Lemma 2.2]{AT} for a proof of \eqref{fixedhorizon}. The constant $n_0$ will make its appearance in the following sections. Here, we want to understand the behaviour for large times of the quadratic forms generated by $C(t)$ and $E(-t)C(t)E(-t)^T$: the reader should keep in mind that the spectral assumption \eqref{hypreal0} is crucial for the validity of the next result.
\begin{proposition}\label{propmatrix} 
There exist positive constants $c_+, c_-$ such that, for every $t\geq 1$, the following inequalities hold
\begin{equation}\label{kappat}
\left\langle E(-t)C(t)E(-t)^T \xi,\xi \right\rangle \geq (c_+)\, t |\xi|^2\quad\forall\,\xi\in\R^N
\end{equation}
and
\begin{equation}\label{ct}
\left\langle C(t) \xi,\xi \right\rangle \geq (c_-)\, t |\xi|^2\quad\forall\,\xi\in\R^N.
\end{equation}
\end{proposition}
\begin{proof}
We first prove the validity of \eqref{ct}. Once this is established, by exchanging the roles of $B\mapsto -B$ we will have also the proof of \eqref{kappat}.\\
We start noticing that our assumptions (as well as the desired conclusion in \eqref{ct}) are invariant under change of bases of $\R^N$ (that is invariance by similarity). Hence, without any loss of generality, we can assume that the matrix $B^T$ coincides with its Jordan's (real) canonical form (see, e.g., \cite[Theorem 3.4.5]{HJ})
\begin{equation}\label{jordan}
B^T=\begin{pmatrix}
J_{n_1} &  &  &  &  &  & 
\\
 & \ddots &  &  & 0 &   & 
\\
 &  & J_{n_q} &  &  &  & \\
 &  &  & C_{m_1}(b_1) &  &  & 
\\
 &  & 0 &  &  & \ddots & 
\\
 &  &  &  &  &   & C_{m_p}(b_p)
\end{pmatrix},
\end{equation}
where by \eqref{hypreal0} the eigenvalues of $B^T$ (keeping in mind that $\sigma(B^T)=\sigma(B)$) are $\{0,\pm ib_1,\ldots, \pm ib_p\}$ with $b_\ell\in \R\smallsetminus\{0\}$, $n_1+\ldots+n_q+2m_1+\ldots+2m_p=N$ with $n_k, m_\ell\in \N$, and the $n_k\times n_k$ matrices $J_{n_k}$ and (respectively) the $2m_\ell\times 2m_\ell$ matrices $C_{m_\ell}(b_\ell)$ are in the following form
$$J_{n_k}=\begin{pmatrix}
0 & 1 & 0 & \ldots & 0  \\
0 & 0 & 1 & \ldots & 0  \\
0 & 0 & \ddots & \ddots & 0  \\
0 & \ldots & 0 & 0 & 1  \\
0 & 0 & \ldots & 0 & 0 
\end{pmatrix},\quad C_{m_\ell}(b_\ell)=\begin{pmatrix}
0 & -b_\ell & 1 & 0 & 0 & \ldots & \ldots & 0  \\
b_\ell & 0 & 0 & 1 & 0 & \ldots & \ldots & 0  \\
0 & 0 & 0 & -b_\ell & 1 & 0 & \ldots & 0  \\
0 & 0 & b_\ell & 0 & 0 & 1 & \ldots & 0 \\
\vdots & \vdots & 0 & 0 & \ddots & \ddots & 1 & 0 \\
\vdots & \vdots & 0 & 0 & \ddots & \ddots & 0 &  1\\
0 & 0 & \ldots & \ldots & 0 & 0 & 0 & -b_\ell \\
0 & 0 & \ldots & \ldots & 0 & 0 & b_\ell & 0 \\
\end{pmatrix}.
$$
Up to a permutation (and a change in sign) of the coordinates, we can always think that $0<b_1\leq b_2\leq \ldots\leq b_p$. For a generic vector $\xi\in \R^N=\R^{n_1}\times\cdots\times\R^{n_q}\times\R^{2m_1}\times\cdots\times\R^{2m_p}=\R^{\bar{n}}\times\R^{2\bar{m}}$ with $\bar{n}=n_1+\ldots+n_q$ and $\bar{m}=m_1+\ldots+m_p$, we use the notation 
$$\xi=(\xi^{(n_1)},\ldots,\xi^{(n_q)},\xi^{(m_1)},\ldots,\xi^{(m_p)})=(\xi^{(\bar{n})},\xi^{(\bar{m})}),$$
as well as the notation $\xi^{(m_\ell)}=( (\xi^{(m_\ell)})_1,\ldots,(\xi^{(m_\ell)})_{m_\ell})$ with $(\xi^{(m_\ell)})_j\in \R^2$. Thanks to the above described upper-diagonal blocks, we have a rather precise information about the non-trivial invariant subspaces for $B^T$. In what follows we shall exploit the following: if $W$ is a non-trivial invariant subspace for $B^T$, then $W$ contains at least either a one-dimensional space $W_0$ or a two-dimensional space $W_{\bar{b}}$ where
\begin{itemize}
\item[{\emph{case 0}})] $W_0$ is generated by an eigenvector of $B^T$ (related to the $0$ eigenvalue), i.e. a non-null vector of the form $\xi_0=(\xi^{(\bar{n})},0)$ with $\xi^{(\bar{n})}=( ((\xi^{(n_1)})_1,0,\ldots,0),\ldots,((\xi^{(n_q)})_1,0,\ldots,0))$;
\item[{\emph{case $\bar{b}$}})] there exist (consecutive) indexes $m_{i_1},\ldots,m_{i_k}$ (for some $k\leq q$) with $b_{i_j}=\bar{b}$ for every $j\in\{1,\ldots,k\}$ such that $W_{\bar{b}}$ contains a non-null vector of the form $\xi_{\bar{b}}=(0,\xi^{(\bar{m})})$ with $\xi^{(\bar{m})}=(\xi^{(m_1)},\ldots,\xi^{(m_p)})$ satisfying $\xi^{(m_\ell)}=0$ for $\ell\notin\{i_1,\ldots,i_k\}$ and $\xi^{(m_{i_j})}=((\xi^{(m_{i_j})})_1,0,\ldots,0)$. In this case the two dimensional space $W_{\bar{b}}$ is spanned by $\xi_{\bar{b}}$ and $B^T\xi_{\bar{b}}$.
\end{itemize}
A direct computation shows that
$$e^{-s B^T}=\begin{pmatrix}
e^{-s J_{n_1}} &  &  &  &  &  & 
\\
 & \ddots &  &  & 0 &   & 
\\
 &  & e^{-s J_{n_q}} &  &  &  & \\
 &  &  & e^{-s C_{m_1}(b_1)} &  &  & 
\\
 &  & 0 &  &  & \ddots & 
\\
 &  &  &  &  &   & e^{-s C_{m_p}(b_p)}
\end{pmatrix},$$
where
$$e^{-s J_{n_k}}=\begin{pmatrix}
1 & -s & \frac{(-s)^2}{2} & \ldots & \frac{(-s)^{n_k-1}}{(n_k-1)!}  \\
0 & 1 & -s & \ldots & \frac{(-s)^{n_k-2}}{(n_k-2)!}  \\
0 & 0 & \ddots & \ddots & \vdots  \\
0 & \ldots & 0 & 1 & -s  \\
0 & 0 & \ldots & 0 & 	1 
\end{pmatrix}$$
and
$$e^{-s C_{m_\ell}(b_\ell)}=\begin{pmatrix}
R_{-s b_\ell} & -s R_{-s b_\ell} & \frac{(-s)^2}{2}R_{-s b_\ell} & \ldots & \frac{(-s)^{m_\ell-1}}{(m_\ell-1)!}R_{-s b_\ell}  \\
0 & R_{-s b_\ell} & -s R_{-s b_\ell} & \ldots & \frac{(-s)^{m_\ell-2}}{(m_\ell-2)!}R_{-s b_\ell}  \\
0 & 0 & \ddots & \ddots & \vdots  \\
0 & \ldots & 0 & R_{-s b_\ell} & -sR_{-s b_\ell}  \\
0 & 0 & \ldots & 0 & 	R_{-s b_\ell} 
\end{pmatrix}$$
with $R_{-s b_\ell}= \begin{pmatrix}
\cos{(s b_\ell)} & \sin{(s b_\ell)}   \\
-\sin{(sb_\ell)} & \cos{(s b_\ell)}
\end{pmatrix}$. We can define, for $r>0$, the function $\delta_r:\R^N\longrightarrow \R^N$ by
$$
\delta_r(\xi)=\left(\delta^{(n_1)}_r\left(\xi^{(n_1)}\right),\ldots,\delta^{(n_q)}_r\left(\xi^{(n_q)}\right),\delta^{(m_1)}_r\left(\xi^{(m_1)}\right),\ldots,\delta^{(m_p)}_r\left(\xi^{(m_p)}\right)\right),
$$
where, for $k\in\{1,\ldots,q\}$ and $\ell\in\{1,\ldots,p\}$, $\delta^{(n_k)}_r:\R^{n_k}\longrightarrow\R^{n_k}$ and $\delta^{(m_\ell)}_r:\R^{2m_\ell}\longrightarrow\R^{2m_\ell}$ are the linear applications acting in the following (diagonal) way
$$
\delta^{(n_k)}_r=\begin{pmatrix}
r & 0 & 0 & \ldots & 0  \\
0 & r^3 & 0 & \ldots & 0  \\
0 & 0 & \ddots & \ddots & 0  \\
0 & \ldots & 0 & r^{2n_k-3} & 0  \\
0 & 0 & \ldots & 0 & r^{2n_k-1} 
\end{pmatrix}\quad\mbox{ and }\quad
\delta^{(m_\ell)}_r=\begin{pmatrix}
r\mathbb{I}_2 & \mathbb{O}_2 & \mathbb{O}_2 & \ldots & \mathbb{O}_2  \\
\mathbb{O}_2 & r^3\mathbb{I}_2 & \mathbb{O}_2 & \ldots & \mathbb{O}_2  \\
\mathbb{O}_2 & \mathbb{O}_2 & \ddots & \ddots & \mathbb{O}_2  \\
\mathbb{O}_2 & \ldots & \mathbb{O}_2 & r^{2m_\ell-3}\mathbb{I}_2 & \mathbb{O}_2  \\
\mathbb{O}_2 & \mathbb{O}_2 & \ldots & \mathbb{O}_2 & r^{2m_\ell-1}\mathbb{I}_2 
\end{pmatrix}.
$$
Under these notations, for any $r>0$ and $\sigma\in\R$, we can show that
\begin{equation}\label{dilinvJ}
e^{-\sigma r^2 J_{n_k}} =  \delta^{(n_k)}_{\frac{1}{r}} e^{-\sigma J_{n_k}}\delta^{(n_k)}_r
\end{equation}
and
\begin{equation}\label{dilinvC}
e^{-\sigma r^2 C_{m_\ell}(b_\ell)} =  \delta^{(m_\ell)}_{\frac{1}{r}} e^{-\sigma C_{m_\ell}(r^2 b_\ell)} \delta^{(m_\ell)}_r.
\end{equation}
Thanks to \eqref{dilinvJ}-\eqref{dilinvC}, for every $\xi\in\R^N$ one has
\begin{align}\label{lets}
\left\langle C(t)\xi,\xi\right\rangle &=\int_0^t \left\langle A e^{-s B^T}\xi,e^{-s B^T}\xi\right\rangle ds = t \int_0^1 \left\langle A e^{-\sigma t B^T}\xi,e^{-\sigma t B^T}\xi\right\rangle d\sigma\\
&= \int_0^1 \left\langle A_t e^{-\sigma B_t} \delta_{\sqrt{t}}(\xi), e^{-\sigma B_t} \delta_{\sqrt{t}}(\xi)\right\rangle d\sigma\notag
\end{align}
where
$$
e^{-\sigma B_t} =\begin{pmatrix}
e^{-\sigma J_{n_1}} &  &  &  &  &  & 
\\
 & \ddots &  &  & 0 &   & 
\\
 &  & e^{-\sigma J_{n_q}} &  &  &  & \\
 &  &  & e^{-\sigma C_{m_1}(tb_1)} &  &  & 
\\
 &  & 0 &  &  & \ddots & 
\\
 &  &  &  &  &   & e^{-\sigma C_{m_p}(tb_p)}
\end{pmatrix},
$$
and
$$
A_t=D_t A D_t
$$
with
$$
D_t=\begin{pmatrix}
\sqrt{t}\delta^{(n_1)}_{\frac{1}{\sqrt{t}}} &  &  &  &  &  & 
\\
 & \ddots &  &  & 0 &   & 
\\
 &  & \sqrt{t}\delta^{(n_q)}_{\frac{1}{\sqrt{t}}} &  &  &  & \\
 &  &  & \sqrt{t}\delta^{(m_1)}_{\frac{1}{\sqrt{t}}} &  &  & 
\\
 &  & 0 &  &  & \ddots & 
\\
 &  &  &  &  &   & \sqrt{t}\delta^{(m_p)}_{\frac{1}{\sqrt{t}}}.
\end{pmatrix}
$$
Let us now define the symmetric and non-negative matrix
$$
A_\infty=D_\infty A D_\infty,
$$
where $D_\infty$ is the diagonal matrix
$$D_\infty=\begin{pmatrix}
D^{(n_1)}_{\infty} &  &  &  &  &  & 
\\
 & \ddots &  &  & 0 &   & 
\\
 &  & D^{(n_q)}_{\infty} &  &  &  & \\
 &  &  & D^{(m_1)}_{\infty} &  &  & 
\\
 &  & 0 &  &  & \ddots & 
\\
 &  &  &  &  &   & D^{(m_p)}_{\infty}
\end{pmatrix}$$
with
$$
D^{(n_k)}_{\infty}=\begin{pmatrix}
1 & 0 & 0 & \ldots & 0  \\
0 & 0 & 0 & \ldots & 0  \\
0 & 0 & \ddots & \ddots & 0  \\
0 & \ldots & 0 & 0 & 0  \\
0 & 0 & \ldots & 0 & 0 
\end{pmatrix}\quad\mbox{ and }\quad
D^{(m_\ell)}_\infty=\begin{pmatrix}
\mathbb{I}_2 & \mathbb{O}_2 & \mathbb{O}_2 & \ldots & \mathbb{O}_2  \\
\mathbb{O}_2 & \mathbb{O}_2 & \mathbb{O}_2 & \ldots & \mathbb{O}_2  \\
\mathbb{O}_2 & \mathbb{O}_2 & \ddots & \ddots & \mathbb{O}_2  \\
\mathbb{O}_2 & \ldots & \mathbb{O}_2 & \mathbb{O}_2 & \mathbb{O}_2  \\
\mathbb{O}_2 & \mathbb{O}_2 & \ldots & \mathbb{O}_2 & \mathbb{O}_2
\end{pmatrix}.
$$
Let us notice that the matrix $D_\infty$ is actually an orthogonal projector as $D^2_\infty=D_\infty$. It is also clear from the definition that
$$
A_t=A_\infty+O(t^{-1})\qquad \mbox{ as }t\to+\infty.
$$
By the previous relation and by noticing that the components of the matrix $e^{-\sigma B_t}$ as functions of the variables $\sigma\in [0,1]$ are uniformly bounded with respect to the parameter $t$, one has that
\begin{equation}\label{reduction}
\int_0^1 \left\langle A_t e^{-\sigma B_t} v, e^{-\sigma B_t}v\right\rangle d\sigma= \int_0^1 \left\langle A_\infty e^{-\sigma B_t} v, e^{-\sigma B_t}v\right\rangle d\sigma + O(t^{-1}) |v|^2
\end{equation}
as $t\to+\infty$ (uniformly with respect to arbitrary $v\in\R^N$). Let us now claim that there exists a constant $\lambda_0>0$ such that
\begin{equation}\label{claimstructure}
\int_0^1 \left\langle A_\infty e^{-\sigma B_t} v, e^{-\sigma B_t}v\right\rangle d\sigma \geq \lambda_0 \int_0^1 \left| D_\infty e^{-\sigma B_t} v \right|^2 d\sigma + O(t^{-1})|v|^2,\,\,\mbox{ as }t\to +\infty.
\end{equation}
For reading convenience, we postpone the lengthy proof of \eqref{claimstructure} to the subSection \ref{subapp} below. We just mention here that, in the non-degenerate case where $A$ is positive definite, then \eqref{claimstructure} is readily satisfied by taking $\lambda_0$ to be the smallest eigenvalue of $A$ since in this case $A_\infty=D_\infty A D_\infty\geq \lambda_0 D_\infty^2$ (no remainder term $O(t^{-1})$ is needed in such situation). Here we give \eqref{claimstructure} for granted, and we proceed in establishing the following crucial bound: there exists a constant $c>0$ such that
\begin{equation}\label{claim0} 
\int_0^1 \left| D_\infty e^{-\sigma B_t} v \right|^2 d\sigma \geq c |v|^2 \quad \mbox{ for all }t>0.
\end{equation}
In order to infer the validity of \eqref{claim0} we first notice that the left-hand side is independent of $t$: as a matter of fact, for $t>0$ and $v\in\R^N$, in our notations we have
\begin{align*}
\left| D_\infty e^{-\sigma B_t} v \right|^2&=\sum_{k=1}^q \left|D_\infty^{(n_k)} e^{-\sigma J_{n_k}} v^{(n_k)}\right|^2 + \sum_{\ell=1}^p \left|D_\infty^{(m_\ell)} e^{-\sigma C_{m_\ell}(tb_\ell)} v^{(m_\ell)}\right|^2\\
&=\sum_{k=1}^q \left|D_\infty^{(n_k)} e^{-\sigma J_{n_k}} v^{(n_k)}\right|^2 + \sum_{\ell=1}^p \left|\sum_{k=1}^{m_\ell} \frac{(-\sigma)^{k-1}}{(k-1)!} R_{-\sigma t b_\ell} \left(v^{(m_\ell)}\right)_k\right|^2\\
&=\sum_{k=1}^q \left|D_\infty^{(n_k)} e^{-\sigma J_{n_k}} v^{(n_k)}\right|^2 + \sum_{\ell=1}^p \left|R_{-\sigma t b_\ell}\left(\sum_{k=1}^{m_\ell} \frac{(-\sigma)^{k-1}}{(k-1)!} \left(v^{(m_\ell)}\right)_k \right)\right|^2\\
&=\sum_{k=1}^q \left|D_\infty^{(n_k)} e^{-\sigma J_{n_k}} v^{(n_k)}\right|^2 + \sum_{\ell=1}^p \left|\sum_{k=1}^{m_\ell} \frac{(-\sigma)^{k-1}}{(k-1)!} \left(v^{(m_\ell)}\right)_k \right|^2
\end{align*}
where in the last equality we used the fact that $R_{-\sigma t b_\ell}$ is a $2\times 2$ orthogonal matrix and it doesn't affect the norms in $\R^2$. Being $t$-independent, the value $\left| D_\infty e^{-\sigma B_t} v \right|^2$ agrees with its value at $t=1$ and we keep in mind that $e^{-\sigma B_1}=e^{-\sigma B^T}$. Hence we have
$$\int_0^1 \left| D_\infty e^{-\sigma B_t} v \right|^2 d\sigma=\int_0^1 \left| D_\infty e^{-\sigma B^T} v \right|^2 d\sigma= \int_0^1 \left\langle  D_\infty e^{-\sigma B^T} v, e^{-\sigma B^T} v\right\rangle d\sigma=:\left\langle C_\infty(1)v,v\right\rangle,$$
where we are denoting by $C_\infty(\cdot)$ the Kalman-type matrix (as the one in \eqref{Kmatrix}) associated to the operator $\Lop_\infty=\tr\left( (D_\infty) D^2\right)+\left\langle Bx,\nabla\right\rangle$. Such operator is degenerate-elliptic (even in the sub-case where $A$ is positive definite) since the diffusion matrix $D_\infty$ might have a very big kernel. Nonetheless, it is remarkable that the hypoellipticity condition still holds for $\Lop_\infty$ thanks to the special mutual structure of the matrices $D_\infty$ and $B$. As a matter of fact, every non-trivial invariant subspace of $B^T$ contains at least one non-null vector of the form $\xi_0$ or $\xi_{\bar{b}}$ described in {\emph{case 0}}) and {\emph{case $\bar{b}$}}) above, and both vectors $\xi_0$ and $\xi_{\bar{b}}$ are spanned by the columns of $D_\infty$: since the columns of $D_\infty$ are orthogonal to $Ker(D_\infty)$, this is implying that $Ker(D_\infty)$ does not contain any non-trivial subspace which is invariant for $B^T$. Such algebraic property for $Ker(D_\infty)$ and the invariant subspaces of $B^T$ is known to be equivalent to the H\"ormander condition (see, e.g., \cite[page 148]{Hor} and \cite[(1.2)]{LP}): this is telling us that the Kalman matrix $C_\infty(\tau)$ is positive definite for any $\tau>0$, and in particular we have the existence of $c>0$ such that
\begin{equation}\label{perdopo}
\int_0^1 \left| D_\infty e^{-\sigma B_t} v \right|^2 d\sigma=\left\langle C_\infty(1)v,v\right\rangle\geq c |v|^2\qquad\forall\,v\in\R^N.
\end{equation}
This fact finishes the proof of \eqref{claim0}. If we put together \eqref{reduction}, \eqref{claimstructure}, and \eqref{claim0}, we obtain
\begin{align*}
&\int_0^1 \left\langle A_t e^{-\sigma B_t} v, e^{-\sigma B_t}v\right\rangle d\sigma\\
&= \int_0^1 \left\langle A_\infty e^{-\sigma B_t} v, e^{-\sigma B_t}v\right\rangle d\sigma + O(t^{-1}) |v|^2\\
&\geq \lambda_0 \int_0^1 \left\langle D_\infty e^{-\sigma B_t} v, D_\infty e^{-\sigma B_t}v\right\rangle d\sigma + O(t^{-1}) |v|^2\geq \lambda_0 c |v|^2 + O(t^{-1}) |v|^2\quad\mbox{ as }t\to +\infty,
\end{align*}
for every $v\in\R^N$. The previous asymptotic relation implies the existence of $T\geq 1$ such that, for every $v\in\R^N$,
\begin{equation}\label{finally} 
\int_0^1 \left\langle A_t e^{-\sigma B_t} v, e^{-\sigma B_t}v\right\rangle d\sigma \geq \frac{\lambda_0 c}{2} |v|^2 \quad \mbox{ for all }t\geq T.
\end{equation}
Finally, we can exploit \eqref{finally} into \eqref{lets} and obtain
\begin{equation}\label{giustaTlarge}
\left\langle C(t)\xi,\xi\right\rangle\geq \frac{\lambda_0 c}{2} \left|\delta_{\sqrt{t}}(\xi)\right|^2\geq \frac{\lambda_0 c}{2} t |\xi|^2 \quad\forall\,t\geq T.
\end{equation}
On the other hand, by the monotonicity properties of $C(t)$, we also have
\begin{equation}\label{giustasmall}
\left\langle C(t)\xi,\xi\right\rangle\geq \left\langle C(1)\xi,\xi\right\rangle\geq \lambda_{C(1)} |\xi|^2=\frac{1}{T}\lambda_{C(1)} T |\xi|^2\geq \frac{1}{T}\lambda_{C(1)} t |\xi|^2 \quad\forall\,t\in [1,T],
\end{equation}
where $\lambda_{C(1)}>0$ is the smallest eigenvalue of the positive definite matrix $C(1)$. From \eqref{giustaTlarge}-\eqref{giustasmall}, we have the desired estimate in \eqref{ct} with the choice
$$
c_-=\min\left\{\frac{\lambda_0 c}{2},\frac{1}{T}\lambda_{C(1)}\right\}.
$$
This concludes the proof of \eqref{ct} and (as we mentioned at the beginning of the proof) also of \eqref{kappat}.
\end{proof}
\begin{corollary}\label{Corlim}Denoting by $c_+, c_-$ the positive constants in Proposition \ref{propmatrix}, for any $t\geq 1$ we have
\begin{equation}\label{kappatinverso}
\left\langle E(t)^TC^{-1}(t)E(t) \xi,\xi \right\rangle \leq \frac{1}{(c_+)\, t} |\xi|^2\quad\forall\,\xi\in\R^N
\end{equation}
and
\begin{equation}\label{ctinverso}
\left\langle C^{-1}(t) \xi,\xi \right\rangle \leq \frac{1}{(c_-)\, t} |\xi|^2 \quad\forall\,\xi\in\R^N.
\end{equation}
In particular, for every $\xi\in\R^N$, one has
\begin{equation}\label{zerolim}
\underset{t\to +\infty}{\lim} \langle C^{-1}(t) \xi, \xi\rangle = \underset{t\to +\infty}{\lim} \left\langle E(t)^TC^{-1}(t)E(t) \xi,\xi \right\rangle =0
\end{equation}
\end{corollary}
\begin{proof}
We recall that, for any symmetric and positive definite matrix $M$ and for any non-null vector $v$, we can consider the unique non-null vector $w$ satisfying $v=\sqrt{M} w$ and we can notice that
$$
\left\langle M^{-1}\frac{v}{|v|},\frac{v}{|v|}\right\rangle=\frac{\left\langle w,w\right\rangle}{|v|^2}=\frac{|w|^2}{\left\langle Mw,w\right\rangle}=\left(\left\langle M\frac{w}{|w|},\frac{w}{|w|}\right\rangle\right)^{-1}.
$$
This yields
\begin{equation}\label{invertire}
\underset{|\zeta|=1}{\sup}\left\langle M^{-1}\zeta,\zeta\right\rangle = \frac{1}{\underset{|\xi|=1}{\inf}\left\langle M\xi,\xi\right\rangle}.
\end{equation}
Exploiting \eqref{invertire} with the choices $M=E(-t)C(t)E(-t)^T$ (whose inverse matrix is given by $M^{-1}=E(t)^TC^{-1}(t)E(t)$) and, respectively, $M=C(t)$ (both matrices are positive definite thanks to \eqref{hyphorm}), we deduce \eqref{kappatinverso}-\eqref{ctinverso} from \eqref{kappat}-\eqref{ct}. The limiting behaviour in \eqref{zerolim} follows easily from \eqref{kappatinverso}-\eqref{ctinverso}.
\end{proof}

The previous algebraic statements have a number of geometric consequences which are crucial for our scopes. The first one concerns the determinant function $D(t)$. On one hand, an immediate consequence of Proposition \ref{propmatrix} is that $D(t)\gtrsim t^N$ for $t\geq 1$ (since each one of the $N$ eigenvalues grows at least as $t$): this is already an improvement with respect to \cite[Proposition 3.1]{GT} where the weaker bound $D(t)\gtrsim t^2$ was established for $N\geq 2$. On the other hand, in the next corollary we show how Proposition \ref{propmatrix} implies in fact a global doubling property for $D(t)$. This phenomenon quantifies the global power-like behaviour of the determinant function.

\begin{corollary}\label{corDoub}
There exists $c_d>1$ such that
\begin{equation}\label{Ddoubling}
D(2t)\leq c_d\, D(t)\qquad\mbox{ for every }t>0.
\end{equation}
In particular
\begin{equation}\label{Vdoubling}
V_p(2t)\leq 2^{\frac{p}{2}} \sqrt{c_d}\, V_p(t)\qquad\mbox{ for every }t>0.
\end{equation}
\end{corollary}
\begin{proof}
From the definition of $C(t)$ or from \eqref{deftKt} one can see that the function $D(t)$ is positive and increasing. Moreover, it is proved in \cite[Lemma 3.1 and Remark 3.1]{LP} that there exist $c_0\geq 1$ and $Q_0\in \N$ such that
$$
\frac{1}{c_0} t^{Q_0}\leq  D(t) \leq c_0 t^{Q_0} \qquad\mbox{ for every }t\in (0,2].
$$
The previous estimate yields
\begin{equation}\label{localdoubling}
D(2t) \leq c_0 2^{Q_0} t^{Q_0}\leq c_0^2 2^{Q_0} D(t) \mbox{ for every }t\in (0,1],
\end{equation}
i.e. the validity of a local doubling condition for the function $D(\cdot)$. We want to deduce from Corollary \ref{Corlim} that such a doubling condition holds in fact globally. To this aim, we recall the so-called Jacobi's formula$$
\frac{d}{dt}\left(\log{\left( \operatorname{det}\left(M(t)\right)\right)}\right)=\operatorname{tr}\left(M^{-1}(t) M'(t)\right),
$$
which is valid for any invertible matrix $M(t)$. If we apply this formula with the choice $M(t)=E(-t)C(t)E(-t)^T$ and we keep in mind that in such case $M'(t)=e^{tB} A e^{t B^T}$ by \eqref{deftKt}, we obtain
$$
\frac{d}{dt}\left(\log{\left( D(t) \right)}\right)=\operatorname{tr}\left( e^{-tB^T} C^{-1}(t) e^{-tB} e^{tB} A e^{t B^T}\right)=\operatorname{tr}\left( C^{-1}(t) A \right).
$$
On the other hand, by exploiting \eqref{ctinverso}, we have
$$
\left\langle C^{-1}(t) \sqrt{A} v, \sqrt{A} v  \right\rangle \leq \frac{1}{(c_-)\, t} |\sqrt{A} v|^2= \frac{1}{(c_-)\, t} \left\langle Av,v\right\rangle
$$
for any vector $v\in \R^N$ and for every $t\geq 1$. This implies that
\begin{equation}\label{gronwally}
\frac{d}{dt}\left(\log{\left( D(t) \right)}\right)=\operatorname{tr}\left( C^{-1}(t) A \right)=\operatorname{tr}\left( \sqrt{A} C^{-1}(t) \sqrt{A} \right) \leq \frac{1}{(c_-)\, t} \operatorname{tr}\left( A\right)\quad\forall\, t\geq 1.
\end{equation}
Hence, for every $t\geq 1$, we obtain
$$
\log{\left( \frac{D(2t)}{D(t)} \right)}=\log{\left( D(2t) \right)}-\log{\left( D(t) \right)}=\int_t^{2t} \frac{d}{d\sigma}\left(\log{\left( D(\sigma) \right)}\right)\,d\sigma\leq \frac{\operatorname{tr}\left( A\right)}{(c_-)} \int_t^{2t} \frac{1}{\sigma}\, d\sigma = \frac{\operatorname{tr}\left( A\right)\log(2)}{(c_-)},
$$
which implies
$$
D(2t)\leq e^{\frac{\operatorname{tr}\left( A\right)\log(2)}{(c_-)}} D(t).
$$
Therefore, we have established the validity of \eqref{Ddoubling} with the choice
$$
c_d:=\max\left\{e^{\frac{\operatorname{tr}\left( A\right)\log(2)}{(c_-)}}, c_0^2 2^{Q_0} \right\}.
$$
Since from the definition of $V_p(\cdot)$ we have
$$
\frac{V_p(2t)}{V_p(t)}=\sqrt{\frac{2^p D(2t)}{D(t)}},
$$
the inequality \eqref{Vdoubling} follows from \eqref{Ddoubling}.
\end{proof}

By a classical argument, from \eqref{Ddoubling} we deduce
$$
\frac{D(R)}{D(r)}\leq c_d\left(\frac{R}{r}\right)^{\log_2{(c_d)}}\quad\mbox{ for any }0<r\leq R.
$$
In particular, we get
\begin{equation}\label{Qp}
\frac{V_p(R)}{V_p(r)}\leq \sqrt{c_d} \left(\frac{R}{r}\right)^{Q_p}\quad\mbox{ for any }0<r\leq R, \quad\mbox{ with }Q_p:=\frac{p+\log_2{(c_d)}}{2}.
\end{equation}

The following theorem represents an enhanced version of Proposition \ref{propmatrix} and Corollary \ref{Corlim} (it boils down to \eqref{kappat} and \eqref{kappatinverso} when the parameter $\mu$ is fixed as $t^{-1}$). It will play a key role in Section \ref{sec5}. It establishes self-similarity properties for the relevant quadratic form at the two different scales $t$ and $\mu t$ in an uniform way with respect to $t\geq 1$ and $0<\mu\leq 1$: the proof of this fact capitalizes further on the main new idea present in the proof of Proposition \ref{propmatrix}, which is the presence of the hypoelliptic operator $\Lop_\infty$ which possesses invariant properties (see, e.g., \eqref{speranza}) and it approximates in some sense $\Lop$ for large times.

\begin{theorem}\label{propmatrix2} There exists a positive constant $k_0$ such that, for every $t\geq 1$ and for every $0<\mu\leq 1$, the following holds
\begin{equation}\label{kappatmaggiorato}
\left\langle E(-t)C(t)E(-t)^T \xi,\xi \right\rangle \geq \frac{k_0}{\mu} \left\langle E(-\mu t)C(\mu t)E(-\mu t)^T \xi,\xi \right\rangle \quad\forall\,\xi\in\R^N.
\end{equation}
In particular, for any $t\geq 1$ and $0<\mu\leq 1$, we have
\begin{equation}\label{kappatmaggioratoinv}
\left\langle E(t)^TC^{-1}(t)E(t) \xi,\xi \right\rangle \leq \frac{\mu}{k_0} \left\langle E(\mu t)^TC^{-1}(\mu t)E(\mu t) \xi,\xi \right\rangle \quad\forall\,\xi\in\R^N.
\end{equation}
\end{theorem}
\begin{proof}
Set $M(t)=E(-t)C(t)E(-t)^T$ (under the holding assumptions, the same estimates hold true also for $M(t)=C(t)$). Let us adopt the same notations and the same strategy of Proposition \ref{propmatrix}: we just replace $B\mapsto -B$ in the proof of Proposition \ref{propmatrix} from the Jordan form \eqref{jordan} for $-B^T$ onwards. In particular, for any $t>0$ and $\xi\in\R^N$, as in \eqref{lets} one has 
\begin{align}\label{letsbis}
\left\langle M(t)\xi,\xi\right\rangle &=\int_0^t \left\langle A e^{s B^T}\xi,e^{s B^T}\xi\right\rangle ds = t \int_0^{1} \left\langle A e^{\sigma t B^T}\xi,e^{\sigma t B^T} \xi\right\rangle d\sigma\\
&= \int_0^{1} \left\langle A_{t} e^{-\sigma B_{t}} \delta_{\sqrt{t}}(\xi), e^{-\sigma B_{t}} \delta_{\sqrt{t}}(\xi)\right\rangle d\sigma.\notag
\end{align}
Recalling that $A_t=A_\infty+O(t^{-1})$ with $A_\infty=D_\infty A D_\infty$, we notice that we have
$$
A_\infty \leq \Lambda_0 D_\infty^2
$$
where $\Lambda_0>0$ is the largest eigenvalue of $A$. Hence, from the previous matrix inequality and from \eqref{claimstructure}-\eqref{perdopo}, there exists a constant $C\geq c>0$ such that we have
\begin{align*}
\left(\lambda_0 c + O(t^{-1})\right)|v|^2 &\leq \lambda_0 \left\langle C_\infty(1)v,v\right\rangle + O(t^{-1})|v|^2\\
&\leq \int_0^{1} \left\langle A_{t} e^{-\sigma B_{t}} v, e^{-\sigma B_{t}} v\right\rangle d\sigma \\
&\leq \Lambda_0\left\langle C_\infty(1)v,v\right\rangle + O(t^{-1})|v|^2 \leq \left(\Lambda_0 C + O(t^{-1})\right)|v|^2
\end{align*}
for every $v\in\R^N$ as $t\to +\infty$. Therefore, there exists $T\geq 1$ such that
$$
\frac{\lambda_0 c}{2}|v|^2 \leq \int_0^{1} \left\langle A_{t} e^{-\sigma B_{t}} v, e^{-\sigma B_{t}} v\right\rangle d\sigma \leq 2\Lambda_0 C|v|^2\quad \forall \, v\in\R^N \mbox{ and for every }t\geq T.
$$
By \eqref{letsbis}, we have thus established the following estimate
\begin{equation}\label{speranza}
\frac{\lambda_0 c}{2}|\delta_{\sqrt{t}}(\xi)|^2\leq \left\langle M(t)\xi,\xi\right\rangle \leq 2\Lambda_0 C|\delta_{\sqrt{t}}(\xi)|^2 \quad \forall \, \xi\in\R^N \mbox{ and for every }t\geq T.
\end{equation}
Let us fix an arbitrary $t\geq T$. If $\frac{T}{t}\leq \mu\leq 1$ we can exploit twice \eqref{speranza} to deduce that
\begin{align*}
\left\langle M(t)\xi,\xi\right\rangle &\geq \frac{\lambda_0 c}{2}|\delta_{\sqrt{t}}(\xi)|^2= \frac{\lambda_0 c}{2}|\delta_{\mu^{-\frac{1}{2}}} \delta_{\sqrt{\mu t}}(\xi)|^2 \geq \frac{\lambda_0 c}{2 \mu}| \delta_{\sqrt{\mu t}}(\xi)|^2 \\
& = \frac{\lambda_0 c}{4\Lambda_0 C \mu} 2\Lambda_0 C| \delta_{\sqrt{\mu t}}(\xi)|^2 \geq \frac{\lambda_0 c}{4\Lambda_0 C \mu} \left\langle M(\mu t)\xi,\xi\right\rangle\quad \forall \, \xi\in\R^N.
\end{align*}
On the other hand, if $0<\mu \leq \frac{T}{t}$, then we can use \eqref{speranza} and \eqref{fixedhorizon} (with $T_0=T$) to infer that
\begin{align*}
\left\langle M(t)\xi,\xi\right\rangle &\geq \frac{\lambda_0 c}{2}|\delta_{\sqrt{t}}(\xi)|^2\geq  \frac{\lambda_0 c}{2} t |\xi|^2 = \frac{\lambda_0 c}{2 K(T) \mu} K(T) \mu t |\xi|^2\\
&  \geq \frac{\lambda_0 c}{2 K(T) \mu} \left\langle M(\mu t)\xi,\xi\right\rangle\quad \forall \, \xi\in\R^N.
\end{align*}
If instead we fix an arbitrary $1\leq t\leq T$, we can use again \eqref{fixedhorizon} and denote by $\lambda_{M(1)}$ the smallest eigenvalue of $M(1)$: in this case, for any $0<\mu\leq 1$, we get
\begin{align*}
\left\langle M(t)\xi,\xi\right\rangle &\geq \left\langle M(1)\xi,\xi\right\rangle \geq \lambda_{M(1)} |\xi|^2 = \frac{\lambda_{M(1)}}{\mu T K(T)}K(T) \mu T |\xi|^2\\
&  \geq \frac{\lambda_{M(1)}}{\mu T K(T)} K(T) \mu t |\xi|^2 \geq \frac{\lambda_{M(1)}}{\mu T K(T)} \left\langle M(\mu t)\xi,\xi\right\rangle\quad \forall \, \xi\in\R^N.
\end{align*}
By putting together the last three estimates we have
$$
\left\langle M(t)\xi,\xi\right\rangle \geq \frac{k_0}{\mu}\left\langle M(\mu t)\xi,\xi\right\rangle\quad \forall \, \xi\in\R^N\mbox{ and for every }\mu\in (0,1],\, t\geq 1
$$
with the choice
$$
k_0=\min\left\{\frac{\lambda_0 c}{4\Lambda_0 C }, \frac{\lambda_0 c}{2 K(T) }, \frac{\lambda_{M(1)}}{T K(T)} \right\}.
$$
This shows \eqref{kappatmaggiorato}. Since $M^{-1}(t)=E(t)^TC^{-1}(t)E(t)$, arguing as in the proof of Corollary \ref{Corlim} we can derive \eqref{kappatmaggioratoinv} from \eqref{kappatmaggiorato}. 
\end{proof}

\subsection{Proof of the claim (\ref{claimstructure})}\label{subapp} In the notations adopted in the proof of Proposition \ref{propmatrix}, we are left with showing the existence of $\lambda_0>0$ such that
\begin{equation}\label{claimstructure2}
\int_0^1 \left\langle \left(  D_\infty A D_\infty \right) e^{-\sigma B_t} v, e^{-\sigma B_t}v\right\rangle d\sigma \geq \lambda_0 \int_0^1 \left| D_\infty e^{-\sigma B_t} v \right|^2 d\sigma + O(t^{-1})|v|^2,\,\,\mbox{ as }t\to +\infty.
\end{equation}
As we mentioned earlier, this is certainly satisfied if $A>0$ (just take $\lambda_0$ to be the smallest eigenvalue of $A$, with the remainder term $O(t^{-1})=0$), and we need to work out the case of $A$ with non-trivial kernel (under \eqref{hyphorm}). To this aim, we need few more notations. Since we have fixed $0<b_1\leq b_2\leq\ldots\leq b_p$, we can distinguish the indexes for which the $b_\ell$'s agree or differ. Let us write
$$
p=p_1+\ldots+p_l
$$
and
$$
\bar{b}_1=\{b_\ell\,:\,\ell\in\{1,\ldots,p_1\}\},\bar{b}_2=\{b_\ell\,:\,\ell\in\{p_1+1,\ldots,p_1+p_2\}\},\ldots,\bar{b}_l=\{b_\ell\,:\,\ell\in\{p-p_l+1,\ldots,p\}\}.
$$
We denote
$$
V_0={\rm{span}}\left\{\xi=(\xi^{(\bar{n})},0)\,:\, \xi^{(\bar{n})}\in \R^{\bar{n}}\right\}
$$
and
\begin{align*}
&V_1={\rm{span}}\left\{\xi=(0,\xi^{(\bar{m})})\,:\,\xi^{(m_\ell)}=0\mbox{ for any }\ell\geq p_1+1\right\},\ldots\\
\ldots,&V_l={\rm{span}}\left\{\xi=(0,\xi^{(\bar{m})})\,:\,\xi^{(m_\ell)}=0\mbox{ for any }\ell\leq p-p_l\right\}.
\end{align*}
Also, we denote by $\Pi_0, \Pi_1,\ldots,\Pi_l$ the orthogonal projectors onto (respectively) the subspaces $V_0, V_1,\ldots, V_l$. Let us keep in mind that by construction 
$$
\mathbb{I}_N=\Pi_0+\Pi_1+\ldots+\Pi_l.
$$

\begin{lemma}
Let $B$ as in \eqref{jordan}, and let $D_\infty$ and $\Pi_0, \Pi_1,\ldots, \Pi_l$ as above. There exist positive constants $\lambda^{(0)},\lambda^{(1)},\ldots,\lambda^{(l)}$ such that
\begin{equation}\label{lambda0}
D_\infty \Pi_0 A \Pi_0 D_\infty \geq \lambda^{(0)} D_\infty \Pi_0 D_\infty
\end{equation}
and, for any $j\in\{1,\ldots,l\}$,
\begin{equation}\label{lambdaj}
D_\infty \Pi_j A \Pi_j D_\infty + D_\infty \Pi_j B A B^T \Pi_j D_\infty \geq \lambda^{(j)} D_\infty \Pi_j D_\infty.
\end{equation}
\end{lemma}
\begin{proof}
Let us first prove \eqref{lambda0}. We notice that $\Pi_0 D_\infty$ is the orthogonal projector fixing the variables $\left(\xi^{(n_1)}\right)_1,\ldots,\left(\xi^{(n_q)}\right)_1$. If we assume by contradiction that \eqref{lambda0} is false, then there exists $\xi_0\in\R^N$ such that $|\Pi_0 D_\infty \xi_0|=1$ and $\left\langle A \Pi_0 D_\infty \xi_0, \Pi_0 D_\infty \xi_0\right\rangle = 0$. Since $A=A^T\geq 0$, this means that $A \Pi_0 D_\infty \xi_0=0$, i.e. $\Pi_0 D_\infty \xi_0 \in Ker(A)$. On the other hand, also $B^T \Pi_0 D_\infty \xi_0=0$. Hence ${\rm{span}}\left\{\Pi_0 D_\infty \xi_0\right\}$ is an invariant subspace for $B^T$ which is contained in $Ker(A)$. This contradicts \eqref{hyphorm}.\\
As far as \eqref{lambdaj} is concerned, let us fix $j\in\{1,\ldots,l\}$. Also in this case, if we assume by contradiction that \eqref{lambdaj} is false, then there exists $\xi_{\bar{b}_j}\in\R^N$ such that $|\Pi_j D_\infty \xi_{\bar{b}_j}|=1$ and
$$
\left\langle A \Pi_j D_\infty \xi_{\bar{b}_j}, \Pi_j D_\infty \xi_{\bar{b}_j}\right\rangle + \left\langle A B^T \Pi_j D_\infty \xi_{\bar{b}_j}, B^T\Pi_j D_\infty \xi_{\bar{b}_j}\right\rangle= 0.
$$
Since $A=A^T\geq 0$, this is saying that both $\Pi_j D_\infty \xi_{\bar{b}_j}$ and $B^T \Pi_j D_\infty \xi_{\bar{b}_j}$ belong to $Ker(A)$. On the other hand, by construction we have $(B^T)^2 \Pi_j D_\infty \xi_{\bar{b}_j}=-\bar{b}_j^2\Pi_j D_\infty \xi_{\bar{b}_j}$. Hence, the two-dimensional space ${\rm{span}}\left\{\Pi_j D_\infty \xi_{\bar{b}_j}, B^T\Pi_j D_\infty \xi_{\bar{b}_j}\right\}$ is an invariant subspace for $B^T$ which is contained in $Ker(A)$. This again contradicts \eqref{hyphorm}.
\end{proof}

We can now set
$$
A_{cut}=  \Pi_0 A \Pi_0 + \Pi_1 A \Pi_1 + \ldots +\Pi_l A \Pi_l
$$
and 
$$
S=D_\infty \left( A-A_{cut} \right) D_\infty.
$$
We notice (from the explicit expression of the matrix $e^{-\sigma B_t}$) that the coefficient appearing in the matrix
$$\int_0^1 e^{-\sigma B_t^T}S e^{-\sigma B_t} d\sigma$$
are finite linear combinations of terms either of the type
$$
\int_0^1 \sigma^i \cos( \sigma t \bar{b}_j) d\sigma\quad\mbox{ and }\quad \int_0^1 \sigma^i \sin( \sigma t \bar{b}_j) d\sigma\qquad \mbox{ for some }i\in\N\cup\{0\},\mbox{ and } j\in\{1,\ldots,l\} 
$$
or of the type
\begin{align*}
&\int_0^1 \sigma^i \cos( \sigma t \bar{b}_{j_1})\cos( \sigma t \bar{b}_{j_2}) d\sigma,\,\, \int_0^1 \sigma^i \sin( \sigma t \bar{b}_{j_1})\cos( \sigma t \bar{b}_{j_2}) d\sigma,\,  \mbox{ and }\, \int_0^1 \sigma^i \sin( \sigma t \bar{b}_{j_1})\sin( \sigma t \bar{b}_{j_2}) d\sigma\\
&\qquad \mbox{ for some }i\in\N\cup\{0\},\mbox{ and some }j_1,j_2\in\{1,\ldots,l\} \mbox{ with }  j_1\neq j_2.
\end{align*}
Since $\bar{b}_j, \bar{b}_{j_1}, \bar{b}_{j_2} >0$ and $\bar{b}_{j_1}\neq \bar{b}_{j_2}$, it is a calculus exercise to see that all the previous integrals are $O(t^{-1})$ as $t\to +\infty$ (basically, one can look at the explicit primitives of $\cos( \sigma t \bar{b}_j)$ and $\cos( \sigma t \bar{b}_{j_1})\cos( \sigma t \bar{b}_{j_2})$). This is showing that
\begin{equation}\label{vialaS}
\int_0^1 \left\langle S e^{-\sigma B_t} v, e^{-\sigma B_t}v\right\rangle d\sigma= O(t^{-1}) |v|^2\quad\forall v\in\R^N,\,\mbox{ as }t\to +\infty
\end{equation}
(where the term $O(t^{-1})$ is uniform with respect to $v$). On the other hand, we have also the following asymptotics as $t\to +\infty$
$$
\begin{cases}
\int_0^1 \sigma^i \cos^2( \sigma t \bar{b}_{j}) d\sigma=\frac{1}{2}\int_0^1 \sigma^i  d\sigma + O(t^{-1}),\\
\int_0^1 \sigma^i \sin^2( \sigma t \bar{b}_{j}) d\sigma=\frac{1}{2}\int_0^1 \sigma^i  d\sigma + O(t^{-1}),\\
\int_0^1 \sigma^i \cos( \sigma t \bar{b}_{j})\sin( \sigma t \bar{b}_{j}) d\sigma= O(t^{-1})
\end{cases}
$$
for every $i\in \N\cup\{0\}$ and $j\in\{1,\ldots,l\}$ (again, since $\bar{b}_j>0$, one can see the decay in $t$ of the primitive functions of $\cos(2\sigma t \bar{b}_{j})$ and $\sin(2\sigma t \bar{b}_{j})$). From these relations, for every $i\in \N\cup\{0\}$ and $j\in\{1,\ldots,l\}$, we have
\begin{equation}\label{keep}
\int_0^1 \sigma^i R^T_{-\sigma t \bar{b}_j}\begin{pmatrix}
\alpha & \beta   \\
\beta & \gamma
\end{pmatrix}R_{-\sigma t \bar{b}_j}d\sigma = \frac{\alpha+\gamma}{2} \left(\int_0^1 \sigma^i  d\sigma\right) \mathbb{I}_2 + O(t^{-1}) \quad\mbox{ as }t\to +\infty
\end{equation}
for any $\alpha$, $\beta$, $\gamma \in\R$. From \eqref{keep} we are going to deduce that
\begin{equation}\label{flippa}
\int_0^1 e^{-\sigma B_t^T}\left(  D_\infty \Pi_j B A B^T \Pi_j D_\infty \right) e^{-\sigma B_t} d\sigma = \bar{b}_j^2 \int_0^1 e^{-\sigma B_t^T}\left(  D_\infty \Pi_j A  \Pi_j D_\infty \right) e^{-\sigma B_t} d\sigma + O(t^{-1}).
\end{equation}
for any $j\in\{1,\ldots,l\}$. To convince ourselves of the validity of this last relation we shall make some explicit computations. To fix the ideas, let us consider the case $j=1$ (the same argument will apply for any $j$) and let us set more notations. We write
$$
D_\infty \Pi_1 A  \Pi_1 D_\infty=\begin{pmatrix}
\mathbb{O}_{\bar{n}} &  \mathbb{O} &  \mathbb{O}  \\ \mathbb{O} &    

\begin{pmatrix}
\tilde{A}_1^{(m_1,m_1)} & \ldots & \tilde{A}_1^{(m_1,m_{p_1})}\\
 & \ddots & \\
\tilde{A}_1^{(m_{p_1},m_1)} &  & \tilde{A}_1^{(m_{p_1},m_{p_1})}
\end{pmatrix}

& \mathbb{O} \\

\mathbb{O} & \mathbb{O} & \mathbb{O}_{2(\bar{m}-m_1-\ldots-m_{p_1})} 
\end{pmatrix}
$$
where, for any $\ell,\ell'\in \{1,\ldots,p_1\}$, the $(2m_\ell)\times(2m_{\ell'})$ matrix $\tilde{A}_1^{(m_\ell,m_{\ell'})}$ is given by
$$\tilde{A}_1^{(m_\ell,m_{\ell'})}=
\begin{pmatrix}
A_1^{(m_\ell,m_{\ell'})} & \mathbb{O}_2 & \mathbb{O}_2 & \ldots & \mathbb{O}_2  \\
\mathbb{O}_2 & \mathbb{O}_2 & \mathbb{O}_2 & \ldots & \mathbb{O}_2  \\
\mathbb{O}_2 & \mathbb{O}_2 & \ddots & \ddots & \mathbb{O}_2  \\
\mathbb{O}_2 & \ldots & \mathbb{O}_2 & \mathbb{O}_2 & \mathbb{O}_2  \\
\mathbb{O}_2 & \mathbb{O}_2 & \ldots & \mathbb{O}_2 & \mathbb{O}_2
\end{pmatrix}
$$
with $A_1^{(m_\ell,m_{\ell'})}$ being the $2\times 2$ matrix formed by coefficients of $A$. Define also $J_{\bar{b}_1}=\begin{pmatrix}
0 & -\bar{b}_1   \\
\bar{b}_1 & 0
\end{pmatrix}$. Then, by applying \eqref{keep} twice and using the fact that $\operatorname{tr}\left(J^T_{\bar{b}_1} A_1^{(m_\ell,m_{\ell'})} J_{\bar{b}_1} \right)=\bar{b}_1^2\operatorname{tr}\left( A_1^{(m_\ell,m_{\ell'})} \right)$, for any $v\in\R^N$ we have
{\allowdisplaybreaks
\begin{align*}
&\int_0^1 \left\langle \left(  D_\infty \Pi_1 A  \Pi_1 D_\infty \right) e^{-\sigma B_t} v, e^{-\sigma B_t}v\right\rangle d\sigma\\
=&\int_0^1 \sum_{\ell, \ell'=1}^{p_1} \left\langle \tilde{A}_1^{(m_\ell,m_{\ell'})} e^{-\sigma C_{m_{\ell'}}(t\bar{b}_1)} v^{(m_{\ell '})} , e^{-\sigma C_{m_\ell}(t\bar{b}_1)} v^{(m_\ell)} \right\rangle d\sigma \\
=&\int_0^1 \sum_{\ell, \ell'=1}^{p_1} \left\langle A_1^{(m_\ell,m_{\ell'})} \left( \sum_{k_1=1}^{m_{\ell'}} \frac{(-\sigma)^{k_1-1}}{(k_1-1)!} R_{-\sigma t \bar{b}_1} \left(v^{(m_{\ell'})}\right)_{k_1}  \right) , \left( \sum_{k=1}^{m_\ell} \frac{(-\sigma)^{k-1}}{(k-1)!} R_{-\sigma t \bar{b}_1} \left(v^{(m_\ell)}\right)_k  \right)\right\rangle d\sigma \\
=&\int_0^1 \sum_{1\leq\ell\leq \ell'\leq p_1} \left\langle \left( R^T_{-\sigma t \bar{b}_1} \frac{A_1^{(m_\ell,m_{\ell'})}+A_1^{(m_{\ell'},m_\ell)}}{2}R_{-\sigma t \bar{b}_1} \right) \left( \sum_{k_1=1}^{m_{\ell'}} \frac{(-\sigma)^{k_1-1}}{(k_1-1)!} \left(v^{(m_{\ell'})}\right)_{k_1}  \right) , \right.\\
&\hspace{8.5cm} \left.\left( \sum_{k=1}^{m_\ell} \frac{(-\sigma)^{k-1}}{(k-1)!} \left(v^{(m_\ell)}\right)_k  \right)\right\rangle d\sigma\\
=&\sum_{1\leq\ell\leq \ell'\leq p_1}\sum_{k=1}^{m_\ell}\sum_{k_1=1}^{m_{\ell'}} \frac{(-1)^{k+k_1}}{(k-1)!(k_1-1)!}\times\\
&\hspace{2cm}\times \int_0^1 \sigma^{k+k_1} \left\langle \left( R^T_{-\sigma t \bar{b}_1} \frac{A_1^{(m_\ell,m_{\ell'})}+A_1^{(m_{\ell'},m_\ell)}}{2} R_{-\sigma t \bar{b}_1} \right)\left(v^{(m_{\ell'})}\right)_{k_1},\left(v^{(m_\ell)}\right)_{k}\right\rangle d\sigma\\
=& \sum_{1\leq\ell\leq \ell'\leq p_1}\sum_{k=1}^{m_\ell}\sum_{k_1=1}^{m_{\ell'}} \frac{(-1)^{k+k_1}}{(k-1)!(k_1-1)!} \frac{\operatorname{tr}\left( A_1^{(m_\ell,m_{\ell'})}+A_1^{(m_{\ell'},m_\ell)}\right)}{4}\times\\
&\hspace{6cm}\times \int_0^1 \sigma^{k+k_1} \left\langle \left(v^{(m_{\ell'})}\right)_{k_1}, \left(v^{(m_\ell)}\right)_{k}\right\rangle d\sigma + O(t^{-1})|v|^2\\
=& \sum_{1\leq\ell\leq \ell'\leq p_1}\sum_{k=1}^{m_\ell}\sum_{k_1=1}^{m_{\ell'}} \frac{(-1)^{k+k_1}}{(k-1)!(k_1-1)!} \frac{\operatorname{tr}\left(J^T_{\bar{b}_1}\left( A_1^{(m_\ell,m_{\ell'})}+A_1^{(m_{\ell'},m_\ell)}\right)J_{\bar{b}_1}\right)}{4\bar{b}_1^2}\times\\
&\hspace{6.2cm}\times \int_0^1 \sigma^{k+k_1} \left\langle \left(v^{(m_{\ell'})}\right)_{k_1}, \left(v^{(m_\ell)}\right)_{k}\right\rangle d\sigma \hspace{-0.1cm}+\hspace{-0.1cm} O(t^{-1})|v|^2\\
=& \frac{1}{\bar{b}_1^2}\sum_{1\leq\ell\leq \ell'\leq p_1}\sum_{k=1}^{m_\ell}\sum_{k_1=1}^{m_{\ell'}} \frac{(-1)^{k+k_1}}{(k-1)!(k_1-1)!} \times\\
&\times \int_0^1 \sigma^{k+k_1} \left\langle \left( R^T_{-\sigma t \bar{b}_1} J^T_{\bar{b}_1}\frac{ A_1^{(m_\ell,m_{\ell'})}+A_1^{(m_{\ell'},m_\ell)} }{2} J_{\bar{b}_1} R_{-\sigma t \bar{b}_1} \right)\left(v^{(m_{\ell'})}\right)_{k_1},\left(v^{(m_\ell)}\right)_{k} \right\rangle d\sigma + O(t^{-1})|v|^2\\
=& \frac{1}{\bar{b}_1^2}\int_0^1 \sum_{1\leq\ell\leq \ell'\leq p_1} \left\langle \left( R^T_{-\sigma t \bar{b}_1} J^T_{\bar{b}_1} \frac{A_1^{(m_\ell,m_{\ell'})}+A_1^{(m_{\ell'},m_\ell)}}{2}J_{\bar{b}_1} R_{-\sigma t \bar{b}_1} \right) \left( \sum_{k_1=1}^{m_{\ell'}} \frac{(-\sigma)^{k_1-1}}{(k_1-1)!} \left(v^{(m_{\ell'})}\right)_{k_1}  \right) ,\right. \\
&\hspace{7.7cm}\left. \left( \sum_{k=1}^{m_\ell} \frac{(-\sigma)^{k-1}}{(k-1)!} \left(v^{(m_\ell)}\right)_k  \right)\right\rangle d\sigma + O(t^{-1})|v|^2\\
=& \frac{1}{\bar{b}_1^2}\int_0^1 \sum_{\ell,\ell'=1}^{p_1} \left\langle \left( J^T_{\bar{b}_1}  A_1^{(m_\ell,m_{\ell'})} J_{\bar{b}_1}  \right) \left( \sum_{k_1=1}^{m_{\ell'}} \frac{(-\sigma)^{k_1-1}}{(k_1-1)!}R_{-\sigma t \bar{b}_1}\left(v^{(m_{\ell'})}\right)_{k_1}  \right) ,\right.\\
&\hspace{5.7cm} \left. \left( \sum_{k=1}^{m_\ell} \frac{(-\sigma)^{k-1}}{(k-1)!} R_{-\sigma t \bar{b}_1}\left(v^{(m_\ell)}\right)_k  \right)\right\rangle d\sigma + O(t^{-1})|v|^2\\
=& \frac{1}{\bar{b}_1^2}\int_0^1 \left\langle \left(  D_\infty \Pi_1 A  \Pi_1 D_\infty \right) B^T \Pi_1 D_\infty e^{-\sigma B_t} v, B^T \Pi_1 D_\infty e^{-\sigma B_t} v\right\rangle d\sigma + O(t^{-1})|v|^2\\
=& \frac{1}{\bar{b}_1^2}\int_0^1 \left\langle \left(  D_\infty \Pi_1 B A B^T  \Pi_1 D_\infty \right) e^{-\sigma B_t} v, e^{-\sigma B_t} v\right\rangle d\sigma + O(t^{-1})|v|^2
\end{align*}
as $t\to+\infty$, where in the last equality we used the fact $B^T \Pi_1 D_\infty\xi= \Pi_1 D_\infty B^T \Pi_1 D_\infty\xi$ for any vector $\xi$. This shows the claimed relation \eqref{flippa}. We are finally ready to complete the proof of \eqref{claimstructure2}. Set
$$
\lambda_0=\min\left\{\lambda^{(0)},\frac{1}{2 \max\{1,\bar{b}^2_1\}}\lambda^{(1)},\ldots,\frac{1}{2\max\{1,\bar{b}^2_l\}}\lambda^{(l)}\right\}>0.
$$
As $t\to+\infty$, for any $v\in\R^N$, by exploiting \eqref{flippa}, \eqref{lambda0}, and \eqref{lambdaj} we have
\begin{align*}
&\int_0^1 \left\langle A_{cut} D_\infty e^{-\sigma B_t} v, D_\infty e^{-\sigma B_t}v\right\rangle d\sigma  \\
=&\int_0^1 \left\langle \left(  D_\infty \Pi_0 A  \Pi_0 D_\infty \right) e^{-\sigma B_t} v, e^{-\sigma B_t}v\right\rangle d\sigma + \sum_{j=1}^{l} \int_0^1 \left\langle \left(  D_\infty \Pi_j A  \Pi_j D_\infty \right) e^{-\sigma B_t} v, e^{-\sigma B_t}v\right\rangle d\sigma \\
=&\int_0^1 \left\langle \left(  D_\infty \Pi_0 A  \Pi_0 D_\infty \right) e^{-\sigma B_t} v, e^{-\sigma B_t}v\right\rangle d\sigma + \frac{1}{2}\sum_{j=1}^{l} \left( \int_0^1 \left\langle \left(  D_\infty \Pi_j A  \Pi_j D_\infty \right) e^{-\sigma B_t} v, e^{-\sigma B_t}v\right\rangle d\sigma +\right.\\
&+\left. \frac{1}{\bar{b}^2_j}\int_0^1 \left\langle \left(  D_\infty \Pi_j B A B^T  \Pi_j D_\infty \right) e^{-\sigma B_t} v, e^{-\sigma B_t}v\right\rangle d\sigma\right) + O(t^{-1})|v|^2\\
\geq & \int_0^1 \left\langle \left(  D_\infty \Pi_0 A  \Pi_0 D_\infty \right) e^{-\sigma B_t} v, e^{-\sigma B_t}v\right\rangle d\sigma +\\
&+ \frac{1}{2}\sum_{j=1}^{l} \frac{1}{\max\{1,\bar{b}^2_j\}} \left( \int_0^1 \left\langle \left(  D_\infty \Pi_j A  \Pi_j D_\infty \right) e^{-\sigma B_t} v, e^{-\sigma B_t}v\right\rangle d\sigma +\right.\\
&+\left. \int_0^1 \left\langle \left(  D_\infty \Pi_j B A B^T  \Pi_j D_\infty \right) e^{-\sigma B_t} v, e^{-\sigma B_t}v\right\rangle d\sigma\right) + O(t^{-1})|v|^2\\
\geq& \lambda^{(0)} \int_0^1 \left\langle \left(  D_\infty \Pi_0 D_\infty \right) e^{-\sigma B_t} v, e^{-\sigma B_t}v\right\rangle d\sigma + \\
&+\frac{1}{2}\sum_{j=1}^{l} \frac{\lambda^{(j)}}{\max\{1,\bar{b}^2_j\}} \int_0^1 \left\langle \left(  D_\infty \Pi_j D_\infty \right) e^{-\sigma B_t} v, e^{-\sigma B_t}v\right\rangle d\sigma + O(t^{-1})|v|^2\\
\geq & \lambda_0 \left( \int_0^1 \left\langle \left(  D_\infty \Pi_0 D_\infty \right) e^{-\sigma B_t} v, e^{-\sigma B_t}v\right\rangle d\sigma + \sum_{j=1}^{l} \int_0^1 \left\langle \left(  D_\infty \Pi_j D_\infty \right) e^{-\sigma B_t} v, e^{-\sigma B_t}v\right\rangle d\sigma\right)+\\
&+ O(t^{-1})|v|^2\\
=&\lambda_0 \int_0^1 \left| D_\infty e^{-\sigma B_t} v \right|^2 d\sigma + O(t^{-1})|v|^2.
\end{align*}
Combining the previous lower bound with \eqref{vialaS}, we obtain
\begin{align*}
& \int_0^1 \left\langle \left(  D_\infty A D_\infty \right) e^{-\sigma B_t} v, e^{-\sigma B_t}v\right\rangle d\sigma \\
=& \int_0^1 \left\langle \left(  D_\infty \left( A_{cut} + A - A_{cut}\right)D_\infty \right) e^{-\sigma B_t} v, e^{-\sigma B_t}v\right\rangle d\sigma\\
=& \int_0^1 \left\langle \left(  D_\infty A_{cut} D_\infty \right) e^{-\sigma B_t} v, e^{-\sigma B_t}v\right\rangle d\sigma + \int_0^1 \left\langle S e^{-\sigma B_t} v, e^{-\sigma B_t}v\right\rangle d\sigma \\
\geq & \lambda_0 \int_0^1 \left| D_\infty e^{-\sigma B_t} v \right|^2 d\sigma + O(t^{-1})|v|^2,
\end{align*}
for any $v\in\R^N$, as $t\to+\infty$. This establishes the desired estimate \eqref{claimstructure2}.
}

\section{Paraboloids}\label{sec4}

For any $z_0=(x_0,t_0)\in \R^{N+1}$ we denote
\begin{align}\label{defpar}
\mathcal{P}_{z_0}=\left\{(x,t)\right. &\in\R^{N+1} \,:\, t<t_0\,\mbox{ and }\\
& \left. \langle C^{-1}(t_0-t) \left(x_0-E(t_0-t)x\right), \left(x_0-E(t_0-t)x\right) \rangle < 1\right\}.\notag
\end{align}

From the results of the previous section it is straightforward to infer the following geometric consequence on the behaviour for large negative times of the paraboloids defined in \eqref{defpar}.

\begin{corollary}\label{lemmapar} For any $z_0\in\R^{N+1}$ and for every $x\in\R^N$, there exists $T=T(x,z_0)<t_0$ such that
$$
(x,t)\in \mathcal{P}_{z_0}\quad\forall\,t<T.
$$
\end{corollary}
\begin{proof}
Recall that, for any symmetric matrix $M\geq 0$, one has 
\begin{equation}\label{doppioprodotto}
0\leq \left\langle M(v+w),v+w \right\rangle\leq 2\left\langle Mv,v\right\rangle + 2\left\langle Mw,w\right\rangle\qquad\forall\, v,w.
\end{equation}
Therefore, for any fixed $z_0\in\R^{N+1}$ and $x\in\R^N$, we have
\begin{align*}
0&\leq \langle C^{-1}(t_0-t) \left(x_0-E(t_0-t)x\right), \left(x_0-E(t_0-t)x\right)\rangle\leq \\
&\leq 2\langle C^{-1}(t_0-t) x_0, x_0\rangle + 2\langle C^{-1}(t_0-t)E(t_0-t)x, E(t_0-t)x\rangle \\
&= 2\langle C^{-1}(t_0-t) x_0, x_0\rangle + 2\langle E(t_0-t)^T C^{-1}(t_0-t)E(t_0-t)x,x\rangle .
\end{align*}
We can thus exploit the property \eqref{zerolim} in Corollary \ref{Corlim} to infer that
$$
\exists\,\underset{t\to-\infty}{\lim}\langle C^{-1}(t_0-t) \left(x_0-E(t_0-t)x\right), \left(x_0-E(t_0-t)x\right) \rangle=0.
$$
Keeping in mind the definition of $\mathcal{P}_{z_0}$ in \eqref{defpar}, this proves the assertion.
\end{proof}

\section{Two onions lemma}\label{sec5}

For any $z_0=(x_0,t_0)\in \R^{N+1}$ and for any $t<t_0$, we denote
\begin{equation}\label{defsigma}
\Sigma_{(t_0-t)}(z_0)=\left\{x\in\R^{N} \,:\, (x,t)\in \mathcal{P}_{z_0} \right\}.
\end{equation}
Then, the following lemma holds where we show a containment property for the onion-shaped sets $\Omega^{(p)}_{r}( z_0 )$.

\begin{lemma}\label{lemmaonion}
For every $p\in\N$ there exists $\theta>1$ such that the following holds: for every $t\leq t_0-1$ and for every $x\in \Sigma_{(t_0-t)}(z_0)$ we have
\begin{equation}\label{onioncontainment}
\Omega^{(p)}_{V_p(t_0-t)}(x,t)\subseteq \Omega^{(p)}_{r^{(p)}_\theta(t_0-t)}( z_0 )
\end{equation}
with the choice
\begin{equation}\label{choicerp}
r^{(p)}_\theta(t_0-t)= \theta\, 2^{\frac{p}{2}} \sqrt{c_d} \, V_p(t_0-t)
\end{equation}
\end{lemma}
\begin{proof}
Fix $p\in\N$, $t\leq t_0-1$, and $x\in \Sigma_{(t_0-t)}(z_0)$. We recall from the definition of $\Sigma_{(t_0-t)}(z_0)$ that this implies
\begin{equation}\label{riscriviSigma}
\langle C^{-1}(t_0-t) \left(x_0-E(t_0-t)x\right), \left(x_0-E(t_0-t)x\right) \rangle < 1
\end{equation}
Fix also an arbitrary point $\zeta=(\xi,\tau)\in \Omega^{(p)}_{V_p(t_0-t)}(x,t)$. This is equivalent of saying that $\tau<t$ and
\begin{align*}
&\exp{ \left(  -\frac{\langle C^{-1}(t-\tau) \left(x-E(t-\tau)\xi\right), \left(x-E(t-\tau)\xi\right) \rangle}{4}   \right)}>\\
&>\frac{(4\pi)^{\frac{N+p}{2}} (t-\tau))^{\frac{p}{2}} \sqrt{\det(C(t-\tau))}}{V_p(t_0-t)} = \frac{V_p(t-\tau)}{V_p(t_0-t)},
\end{align*}
that is
\begin{equation}\label{zetainOm}
\frac{1}{2}\langle C^{-1}(t-\tau) \left(x-E(t-\tau)\xi\right), \left(x-E(t-\tau)\xi\right) \rangle < 2 \log{\left(\frac{V_p(t_0-t)}{V_p(t-\tau)}\right)}.
\end{equation}
We are going to prove that
\begin{equation}\label{zetainaltroOm}
\frac{1}{2}\langle C^{-1}(t_0-\tau) \left(x_0-E(t_0-\tau)\xi\right), \left(x_0-E(t_0-\tau)\xi\right) \rangle < 2 \log{\left(\frac{r^{(p)}_\theta(t_0-t)}{V_p(t_0-\tau)}\right)}
\end{equation}
under the choice \eqref{choicerp} with
\begin{equation}\label{choicetheta}
\log{(\theta)}\geq \frac{1}{2}+\frac{2Q_p}{k_0}\max_{s\in (0,1/2]}\left\{s  \log{\left(c_d^{\frac{1}{2Q_p}} \left(\frac{1}{s}-1\right) \right)}\right\}.
\end{equation}
The inequality \eqref{zetainaltroOm} is saying that $\zeta\in \Omega^{(p)}_{r^{(p)}_\theta(t_0-t)}( z_0 )$: once this is established, the proof of the desired statement will be complete. Let us thus show \eqref{zetainaltroOm}. As we keep in mind that $\tau<t<t_0$, we denote
$$
\mu:=\frac{t-\tau}{t_0-\tau}>0.
$$
Since $\zeta$ satisfies \eqref{zetainOm} and $C^{-1}(t-\tau)$ is positive definite, we have in particular that $V_p(t_0-t)>V_p(t-\tau)$: since the function $V_p(\cdot)$ is increasing, we infer that $t_0-t>t-\tau$. Hence, we have
$$
1<\frac{t_0-t}{t-\tau}=\frac{1}{\mu}-1 \qquad \Longrightarrow \qquad \mu\in \left(0,\frac{1}{2}\right).
$$
Then, as $\frac{1}{1-\mu}<2$, by \eqref{choicerp} and the doubling condition \eqref{Vdoubling} we obtain
\begin{equation}\label{1lower}
\frac{r^{(p)}_\theta(t_0-t)}{V_p(t_0-\tau)}=\frac{r^{(p)}_\theta(t_0-t)}{V_p(\frac{1}{1-\mu}(t_0-t))}> \frac{r^{(p)}_\theta(t_0-t)}{V_p(2(t_0-t))}\geq \frac{r^{(p)}_\theta(t_0-t)}{2^{\frac{p}{2}} \sqrt{c_d} V_p((t_0-t))}\geq \theta.
\end{equation}
Also, since $\frac{\mu}{1-\mu}<1$, we can use \eqref{Qp} to deduce
\begin{equation}\label{2upper}
\frac{V_p(t_0-t)}{V_p(t-\tau)}=\frac{V_p(t_0-t)}{V_p(\frac{\mu}{1-\mu} (t_0-t))}\leq \sqrt{c_d} \left(\frac{t_0-t}{\frac{\mu}{1-\mu} (t_0-t)}\right)^{Q_p} = \sqrt{c_d} \left(\frac{1}{\mu}-1\right)^{Q_p}.
\end{equation}
Let us now consider the left-hand side in \eqref{zetainaltroOm}. By using \eqref{doppioprodotto} we have
\begin{align}\label{daquisomma}
&\frac{1}{2}\langle C^{-1}(t_0-\tau) \left(x_0-E(t_0-\tau)\xi\right), \left(x_0-E(t_0-\tau)\xi\right) \rangle \\
&\leq \langle C^{-1}(t_0-\tau) \left(x_0-E(t_0-t)x\right), \left(x_0-E(t_0-t)x\right) \rangle +\notag\\
&+ \langle C^{-1}(t_0-\tau) \left(E(t_0-t)x-E(t_0-\tau)\xi\right), \left(x_0-E(t_0-t)x\right) \rangle \notag\\
&= \langle C^{-1}(t_0-\tau) \left(x_0-E(t_0-t)x\right), \left(x_0-E(t_0-t)x\right) \rangle +\notag\\
&+ \langle \left( E(t_0-t)^TC^{-1}(t_0-\tau)E(t_0-t) \right)\left(x-E(t-\tau)\xi\right), \left(x-E(t-\tau)\xi\right) \rangle.\notag
\end{align}
On one hand, since $C(t_0-\tau)\geq C(t_0-t)$ we have $C^{-1}(t_0-\tau)\leq C^{-1}(t_0-t)$ (keep in mind \eqref{invertire}) and then
\begin{align}\label{meno1}
&\langle C^{-1}(t_0-\tau) \left(x_0-E(t_0-t)x\right), \left(x_0-E(t_0-t)x\right) \rangle\\
&\leq \langle C^{-1}(t_0-t) \left(x_0-E(t_0-t)x\right), \left(x_0-E(t_0-t)x\right) \rangle <1,\notag
\end{align}
where the last inequality comes from \eqref{riscriviSigma}. On the other hand, since $t_0-\tau\geq t_0-t\geq 1$, we can exploit \eqref{kappatmaggioratoinv} in order to have
$$
E(t_0-\tau)^TC^{-1}(t_0-\tau)E(t_0-\tau)  \leq \frac{\mu}{k_0}  E(\mu (t_0-\tau))^TC^{-1}(\mu (t_0-\tau))E(\mu (t_0-\tau)),
$$
or equivalently
$$
E(t_0-t)^TC^{-1}(t_0-\tau)E(t_0-t)  \leq \frac{\mu}{k_0}  C^{-1}(t-\tau).
$$
If we use the previous matrix inequality, together with the assumption \eqref{zetainOm} and the bound \eqref{2upper} we deduce
\begin{align}\label{menodaje}
&\langle \left( E(t_0-t)^TC^{-1}(t_0-\tau)E(t_0-t) \right)\left(x-E(t-\tau)\xi\right), \left(x-E(t-\tau)\xi\right)  \rangle\\
&\leq \frac{\mu}{k_0} \langle C^{-1}(t-\tau)\left(x-E(t-\tau)\xi\right), \left(x-E(t-\tau)\xi\right)  \rangle\notag \\
&\leq \frac{4\mu}{k_0}  \log{\left(\frac{V_p(t_0-t)}{V_p(t-\tau)}\right)}\leq \frac{4Q_p}{k_0} \mu  \log{\left(c_d^{\frac{1}{2Q_p}} \left(\frac{1}{\mu}-1\right) \right)}.\notag
\end{align}
If we plug \eqref{meno1} and \eqref{menodaje} into \eqref{daquisomma} and we use the choice in \eqref{choicetheta}, we obtain
\begin{align*}
&\frac{1}{2}\langle C^{-1}(t_0-\tau) \left(x_0-E(t_0-\tau)\xi\right), \left(x_0-E(t_0-\tau)\xi\right) \rangle \\
&<1+\frac{4Q_p}{k_0} \mu  \log{\left(c_d^{\frac{1}{2Q_p}} \left(\frac{1}{\mu}-1\right) \right)}\leq 2\log{(\theta)}\\
&\leq 2 \log{\left(\frac{r^{(p)}_\theta(t_0-t)}{V_p(t_0-\tau)}\right)}
\end{align*}
where in the final inequality we have exploited the lower bound in \eqref{1lower}. This completes the proof of \eqref{zetainaltroOm}, and of the lemma.
\end{proof}

\section{Kernel bounds}\label{sec6}
Aim of this section is to prove the following lemma, the main step toward the proof of Theorem \ref{ancient}. Let $p\in\N$, $p>2+4n_0$,  where $n_0$ is the natural number at the last left hand side of \eqref{fixedhorizon}. Then, the following lemma holds

\begin{lemma}\label{cinqueuno} Let $p\in \N$ be as above. There exist $\bar{\theta}>1$ and $c>0$ such that the following holds: for every $t\leq t_0-1$ and for every $x\in \Sigma_{(t_0-t)}(z_0)$ we have
$$
\frac{W_{r^{(p)}_{\bar{\theta}}(t_0-t)}^{(p)}(z_0;\zeta)}{W_{V_p(t_0-t)}^{(p)}(z;\zeta)}\geq c\qquad \mbox{ for all }\zeta\in \Omega^{(p)}_{V_p(t_0-t)}(x,t),
$$
where $r^{(p)}_{\bar{\theta}}(t_0-t)$ is fixed as in \eqref{choicerp}.
\end{lemma}
\begin{proof}
Let us fix
\begin{align}\label{choicebartheta}
\bar{\theta}=2 \theta,\qquad &\mbox{ where $\theta$ is the one in Lemma \ref{lemmaonion}}\\
&\mbox{ for the choice of $p$ fixed at the beginning of this section.}\notag
\end{align}
Let $t\leq t_0-1$, $x\in \Sigma_{(t_0-t)}(z_0)$, and $\zeta=(\xi,\tau)\in \Omega^{(p)}_{V_p(t_0-t)}(x,t)$. If we keep in mind the definitions of $W^{(p)}_r$ and $R_r$, by \eqref{onioncontainment} and \eqref{choicerp} we have
\begin{align*}
&W_{r^{(p)}_{\bar{\theta}}(t_0-t)}^{(p)}(z_0;\zeta)\geq \frac{p \omega_p}{4(p+2)} \frac{R^{p+2}_{r^{(p)}_{\bar{\theta}}(t_0-t)}(z_0;\zeta)}{(t_0-\tau)^2}\\
&\geq \frac{p \omega_p 2^{p+2}}{4(p+2)}(t_0-\tau)^{\frac{p+2}{2}-2}\left(\log{\left(r^{(p)}_{\bar{\theta}}(t_0-t) (4\pi (t_0-\tau))^{-\frac{p}{2}} \Gamma(z_0;\zeta) \right)}\right)^{\frac{p+2}{2}}\\
&= \frac{p \omega_p 2^{p}}{p+2}(t_0-\tau)^{\frac{p-2}{2}}\left(\log{ \left( \frac{\bar{\theta}}{\theta}  r^{(p)}_{\theta}(t_0-t) (4\pi (t_0-\tau))^{-\frac{p}{2}} \Gamma(z_0;\zeta) \right)}\right)^{\frac{p+2}{2}}\\
&> \frac{p \omega_p 2^{p}}{p+2}(t_0-\tau)^{\frac{p-2}{2}}\log^{\frac{p+2}{2}}{ \left( \frac{\bar{\theta}}{\theta}\right)}.
\end{align*}
Since $p\geq 2$ and $\tau <t$, the previous inequality yields
\begin{equation}\label{lowertheta}
W_{r^{(p)}_{\bar{\theta}}(t_0-t)}^{(p)}(z_0;\zeta)\geq \frac{p \omega_p 2^{p}}{p+2}(\log{2})^{\frac{p+2}{2}}(t_0-t)^{\frac{p-2}{2}},
\end{equation}
where we used the choice \eqref{choicebartheta}. As far as the upper bound is concerned we first write
$$
W_{V_p(t_0-t)}^{(p)}(z;\zeta)=\omega_p R^p_{V_p(t_0-t)}(z;\zeta) W(z;\zeta) + \frac{p \omega_p}{4(p+2)} \frac{R^{p+2}_{V_p(t_0-t)}(z;\zeta)}{(t-\tau)^2}.
$$
From the definition of $V_p$ we have
\begin{align}\label{faciledai}
\frac{p \omega_p}{4(p+2)} \frac{R^{p+2}_{V_p(t_0-t)}(z;\zeta)}{(t-\tau)^2}&=
\frac{p \omega_p 2^{p+2}}{4(p+2)}(t-\tau)^{\frac{p+2}{2}-2}\left(\log{\left(V_p(t_0-t) (4\pi (t-\tau))^{-\frac{p}{2}} \Gamma(z;\zeta) \right)}\right)^{\frac{p+2}{2}}\\
&\leq \frac{p \omega_p 2^{p}}{p+2}(t-\tau)^{\frac{p-2}{2}}\left(\log{\left( \frac{V_p(t_0-t)}{V_p(t-\tau)} \right)}\right)^{\frac{p+2}{2}}.\notag
\end{align}
On the other hand, denoting by $\Lambda_A$ the largest eigenvalue of $A$, we have

\begin{align}\label{bivio}
&\omega_p R^p_{V_p(t_0-t)}(z;\zeta) W(z;\zeta)\\
&=\omega_p 2^{p} (t-\tau)^{\frac{p}{2}} \left(\log{\left(V_p(t_0-t) (4\pi (t-\tau))^{-\frac{p}{2}} \Gamma(z;\zeta) \right)}\right)^{\frac{p}{2}}W(z;\zeta)\notag\\
&\leq \frac{\Lambda_A \omega_p 2^{p}}{4}(t-\tau)^{\frac{p}{2}}\left(\log{\left( \frac{V_p(t_0-t)}{V_p(t-\tau)} \right)}\right)^{\frac{p}{2}} \left| \left(E(\tau-t)C(t-\tau)E(\tau-t)^T\right)^{-1}\left(\xi-E(\tau-t)x\right)\right|^2\notag
\end{align}
With the analogous notation for $\Lambda_M$, we now use the fact that
\begin{equation}\label{doppio}
\left| M v\right|^2=\left\langle M \sqrt{M} v, \sqrt{M}v \right\rangle\leq \Lambda_M \left| \sqrt{M} v\right|^2=\Lambda_M\left\langle Mv,v\right\rangle
\end{equation}
for any symmetric non-negative definite $M$. In case $t-\tau\geq 1$, we can use \eqref{kappatinverso} into \eqref{doppio} with the choice $M=\left(E(\tau-t)C(t-\tau)E(\tau-t)^T\right)^{-1}$ and plug this estimate in \eqref{bivio}: this yields
\begin{align*}
&\omega_p R^p_{V_p(t_0-t)}(z;\zeta) W(z;\zeta)\\
&\leq\frac{\Lambda_A \omega_p 2^{p}}{4 c_+} (t-\tau)^{\frac{p-2}{2}}\left(\log{\left( \frac{V_p(t_0-t)}{V_p(t-\tau)} \right)}\right)^{\frac{p}{2}} \left\langle  \left(E(\tau-t)C(t-\tau)E(\tau-t)^T\right)^{-1}\left(\xi-E(\tau-t)x\right),\right.\\
&\hspace{12.1cm}\left.\vphantom{\left(E(\tau-t)C(t-\tau)E(\tau-t)^T\right)^{-1}}\xi-E(\tau-t)x\right\rangle\\
&=\frac{\Lambda_A \omega_p 2^{p}}{4 c_+} (t-\tau)^{\frac{p-2}{2}}\left(\log{\left( \frac{V_p(t_0-t)}{V_p(t-\tau)} \right)}\right)^{\frac{p}{2}} \left\langle C^{-1}(t-\tau)\left(x-E(t-\tau)\xi\right),x-E(t-\tau)\xi\right\rangle\\
&< \frac{\Lambda_A \omega_p 2^{p}}{c_+} (t-\tau)^{\frac{p-2}{2}}\left(\log{\left( \frac{V_p(t_0-t)}{V_p(t-\tau)} \right)}\right)^{\frac{p+2}{2}},
\end{align*}
where in the last inequality we have used \eqref{zetainOm} (as we have fixed $\zeta\in \Omega^{(p)}_{V_p(t_0-t)}(x,t)$). Combining the last estimate with \eqref{faciledai} we have
\begin{equation}\label{upperVp1}
W_{V_p(t_0-t)}^{(p)}(z;\zeta)\leq C_p  (t-\tau)^{\frac{p-2}{2}}\left(\log{\left( \frac{V_p(t_0-t)}{V_p(t-\tau)} \right)}\right)^{\frac{p+2}{2}}
\end{equation}
where
$$
C_p= \omega_p 2^{p}\left(\frac{\Lambda_A }{c_+}+ \frac{p}{p+2}\right).
$$
Let us manipulate further the upper bound in \eqref{upperVp1}. To this aim, we keep the same notations as in Lemma \ref{lemmaonion} by letting
$$
\mu:=\frac{t-\tau}{t_0-\tau}\in (0,\frac{1}{2})\qquad\Longrightarrow\qquad \frac{\mu}{1-\mu}=\frac{t-\tau}{t_0-t}\in (0,1).
$$
By using again \eqref{2upper}, from \eqref{upperVp1} we obtain
\begin{align*}
W_{V_p(t_0-t)}^{(p)}(z;\zeta)&\leq C_p (t_0-t)^{\frac{p-2}{2}} \left( \frac{t-\tau}{t_0-t} \right)^{\frac{p-2}{2}}\left(\log{\left( \frac{V_p(t_0-t)}{V_p(t-\tau)} \right)}\right)^{\frac{p+2}{2}}\\
&\leq C_p (t_0-t)^{\frac{p-2}{2}} \left( \frac{\mu}{1-\mu} \right)^{\frac{p-2}{2}} \left(\log{\left( \sqrt{c_d} \left(\frac{1-\mu}{\mu}\right)^{Q_p} \right)}\right)^{\frac{p+2}{2}}\\
& = C_p Q_p^{\frac{p+2}{2}} (t_0-t)^{\frac{p-2}{2}} \left( \frac{\mu}{1-\mu} \right)^{\frac{p-2}{2}} \left(\log{\left( c_d^{\frac{1}{2Q_p}} \frac{1-\mu}{\mu} \right)}\right)^{\frac{p+2}{2}}.
\end{align*}
By letting
$$
M_p:=\max_{s\in (0,1)}\left\{s^{\frac{p-2}{2}}  \left(\log{\left(c_d^{\frac{1}{2Q_p}} s^{-1} \right)}\right)^{\frac{p+2}{2}}\right\}
$$
(which is finite since $p>2$), we finally get
\begin{equation}\label{upperVp2}
W_{V_p(t_0-t)}^{(p)}(z;\zeta) \leq M_p C_p Q_p^{\frac{p+2}{2}} (t_0-t)^{\frac{p-2}{2}}.
\end{equation}
This was the case where $t-\tau\geq 1$. In case $0<t-\tau< 1$, we can instead use \eqref{fixedhorizon} (with $T_0=1$) into \eqref{doppio} with the same choice as before $M=\left(E(\tau-t)C(t-\tau)E(\tau-t)^T\right)^{-1}$ and plug this different estimate in \eqref{bivio}: this yields
\begin{align*}
&\omega_p R^p_{V_p(t_0-t)}(z;\zeta) W(z;\zeta)\\
&\leq\frac{\Lambda_A \omega_p 2^{p} K(1)}{4} (t-\tau)^{\frac{p-2-4n_0}{2}}\left(\log{\left( \frac{V_p(t_0-t)}{V_p(t-\tau)} \right)}\right)^{\frac{p}{2}} \left\langle C^{-1}(t-\tau)\left(x-E\omega_p(t-\tau)\xi\right),  x-E(t-\tau)\xi\right\rangle\\
&< \Lambda_A \omega_p 2^{p} K(1)(t-\tau)^{\frac{p-2-4n_0}{2}}\left(\log{\left( \frac{V_p(t_0-t)}{V_p(t-\tau)} \right)}\right)^{\frac{p+2}{2}}.
\end{align*}
Combining the last estimate with \eqref{faciledai} and letting 
$$
C'_p=\omega_p 2^{p}\max\left\{\frac{p}{p+2}, \Lambda_A K(1)\right\},
$$
we now have
\begin{align}\label{upperVp1bis}
&W_{V_p(t_0-t)}^{(p)}(z;\zeta)\leq C'_p\left(  (t-\tau)^{\frac{p-2}{2}} +(t-\tau)^{\frac{p-2-4n_0}{2}} \right)\left(\log{\left( \frac{V_p(t_0-t)}{V_p(t-\tau)} \right)}\right)^{\frac{p+2}{2}}\\
&=C'_p\left(  (t_0-t)^{\frac{p-2}{2}}\left(\frac{t-\tau}{t_0-t}\right)^{\frac{p-2}{2}} +(t_0-t)^{\frac{p-2-4n_0}{2}}\left(\frac{t-\tau}{t_0-t}\right)^{\frac{p-2-4n_0}{2}} \right)\left(\log{\left( \frac{V_p(t_0-t)}{V_p(t-\tau)} \right)}\right)^{\frac{p+2}{2}}\notag\\
&\leq C'_p(t_0-t)^{\frac{p-2}{2}} \left(\left(\frac{t-\tau}{t_0-t}\right)^{\frac{p-2}{2}} +\left(\frac{t-\tau}{t_0-t}\right)^{\frac{p-2-4n_0}{2}} \right)\left(\log{\left( \frac{V_p(t_0-t)}{V_p(t-\tau)} \right)}\right)^{\frac{p+2}{2}}\notag
\end{align}
where in the last inequality we used that $t_0-t\geq 1$. We can now argue as before and deduce that
\begin{align*}
&W_{V_p(t_0-t)}^{(p)}(z;\zeta)\leq C'_p(t_0-t)^{\frac{p-2}{2}}\left( \left( \frac{\mu}{1-\mu} \right)^{\frac{p-2}{2}} + \left( \frac{\mu}{1-\mu} \right)^{\frac{p-2-4n_0}{2}} \right) \left(\log{\left( \sqrt{c_d} \left(\frac{1-\mu}{\mu}\right)^{Q_p} \right)}\right)^{\frac{p+2}{2}}\\
&=C'_p Q_p^{\frac{p+2}{2}} (t_0-t)^{\frac{p-2}{2}}\left( \left( \frac{\mu}{1-\mu} \right)^{\frac{p-2}{2}} + \left( \frac{\mu}{1-\mu} \right)^{\frac{p-2-4n_0}{2}} \right) \left(\log{\left( c_d^{\frac{1}{2Q_p}} \frac{1-\mu}{\mu} \right)}\right)^{\frac{p+2}{2}}.
\end{align*}
By letting
$$
M'_p:=\max_{s\in (0,1)}\left\{s^{\frac{p-2-4n_0}{2}}  \left(\log{\left(c_d^{\frac{1}{2Q_p}} s^{-1} \right)}\right)^{\frac{p+2}{2}}\right\}
$$
(which is finite since $p>2+4n_0$), we finally get
\begin{equation}\label{upperVp2bis}
W_{V_p(t_0-t)}^{(p)}(z;\zeta) \leq (M_p + M'_p) C'_p Q_p^{\frac{p+2}{2}} (t_0-t)^{\frac{p-2}{2}}
\end{equation}
also in the case $0<t-\tau< 1$. If we compare \eqref{upperVp2} and \eqref{upperVp2bis} we realize that (regardless of the time coordinate of $\zeta$) we have
$$
W_{V_p(t_0-t)}^{(p)}(z;\zeta) \leq (M_p + M'_p) \max\{C_p,C'_p\} Q_p^{\frac{p+2}{2}} (t_0-t)^{\frac{p-2}{2}}.
$$
By putting together the last upper bound with \eqref{lowertheta} we finally deduce
$$
\frac{W_{r^{(p)}_{\bar{\theta}}(t_0-t)}^{(p)}(z_0;\zeta)}{W_{V_p(t_0-t)}^{(p)}(z;\zeta)}\geq  \frac{p \omega_p 2^{p}}{(p+2)(M_p + M'_p) \max\{C_p,C'_p\}}\left(\frac{\log{2}}{Q_p}\right)^{\frac{p+2}{2}}=:c
$$
as desired.
\end{proof}

\section{Proof of Theorem \ref{ancient}}\label{sec7}
Let $u$ be a smooth bounded below solution to $\cappa u=0$ in $\erreu$. Then 
$$m:=\inf_\erreu u$$
is a real number and 
$$\hat u:=u-m$$
is a smooth non-negative solution to $\cappa \hat u=0$ in $\erreu$ such that 
$$\inf_\erreu \hat u=0.$$
Then, for every $\varepsilon >0$ there exists $z_\eps=(x_\eps,t_\eps)\in\erreu$ such that
\begin{equation} \label{eps} \ucap(z_\eps)<\eps.
\end{equation}
Let us now consider the paraboloid $\Par_{z_\eps}$ defined in \eqref{defpar}. By Corollary \ref{lemmapar}, for every fixed $x\in\erreu$ there exists $T=(x,z_\eps)<t_\eps$ such that 
\begin{equation}\label{7.2} z:=(x,t)\in \Par_{z_\eps}\quad \forall \ t<T.
\end{equation}
We may (and we do) assume $T<t_\eps -1$. We want to prove the following {\it Harnack-type inequality}:\\ there exists a constant $c^\ast >0$, independent of $\eps$, such that 
\begin{equation}\label{harnackeps} \ucap(z_\eps)\geq c^*\ucap(z) \quad \forall\  z=(x,t) \,\mbox{ with }  t<T.
\end{equation}
We stress that, by \eqref{7.2}, condition $t<T$ implies $(x,t)\in \Par_{z_\eps}.$ By the Mean Value formula \eqref{mvf} we have 
\begin{equation} \label{settequattro} \ucap(z_\eps)= \frac{1}{r^{(p)}_{\bar{\theta}}(t_\eps-t)}\int_{\Omega^{(p)}_{r^{(p)}_{\bar{\theta}}(t_\eps-t)}(z_\eps)} \ucap (\zeta) W_{r^{(p)}_{\bar{\theta}}(t_\eps-t)}^{(p)}(z_\eps;\zeta) d\zeta,
\end{equation}
where $r^{(p)}_{\bar{\theta}}(\cdot)$ is given in Lemma \ref{cinqueuno}. Keep in mind that we have chosen $p\in\mathbb{N}$ such that 
$p>2+4n_0$, as in Lemma \ref{cinqueuno}. We also stress that, by our construction, we have $t\leq t_\eps-1$ and $x\in \Sigma_{(t_\eps-t)}(z_\eps)$. Hence, by Lemma \ref{lemmaonion} (and \eqref{choicebartheta}) we have the inclusion
\begin{equation*}
\Omega^{(p)}_{V_p(t_\eps-t)}(x,t)\subseteq \Omega^{(p)}_{\frac{1}{2}r^{(p)}_{\bar{\theta}}(t_\eps-t)}(z_\eps)\subseteq \Omega^{(p)}_{r^{(p)}_{\bar{\theta}}(t_\eps-t)}(z_\eps),
\end{equation*}
being $V_p(\cdot)$ defined in \eqref{defV}. Then, since $\ucap\geq 0,$ from \eqref{settequattro} we get 
\begin{equation*} \ucap(z_\eps)\geq \frac{1}{r^{(p)}_{\bar{\theta}}(t_\eps-t)}\int_{\Omega^{(p)}_{V_p(t_\eps-t)}(x,t)} \ucap (\zeta) W_{r^{(p)}_{\bar{\theta}}(t_\eps-t)}^{(p)}(z_\eps;\zeta) d\zeta.
\end{equation*}
This implies, by \eqref{choicerp}, by Lemma \ref{cinqueuno}, and again by the Mean Value formula, that
\begin{eqnarray*} \ucap(z_\eps)&\geq& c^* \frac{1}{V_p(t_\eps-t)}\int_{\Omega^{(p)}_{V_p(t_\eps-t)}(x,t)} \ucap (\zeta) W_{V_p(t_\eps-t)}^{(p)}(z;\zeta) d\zeta\\&=&c^*\ucap(z).
\end{eqnarray*}
Here $c^*=\frac{c}{\bar{\theta} 2^{\frac{p}{2}}\sqrt{c_d}}$, where $c$ and $\bar{\theta}$ are the absolute constants in Lemma \ref{cinqueuno}, and $c_d$ is the absolute constants in Corollary \ref{corDoub}. Thus, \eqref{harnackeps} is proved.

The inequality \eqref{harnackeps}, together with \eqref{eps}, gives 
\begin{eqnarray*} \ucap(x,t)< \frac{\eps}{c^*}.\end{eqnarray*}
Summing up, we have established the following:\\
for every $\eps>0$ and for every $x\in\erren$, there exists $T\in\erre$, depending on $\eps$ and $x$, such that
$$0\le \ucap(x,t)<\eps \mbox{ for } t<T.$$
Hence,
$$\lim_{t\to -\infty} \ucap(x,t)=0 \quad \forall\ x\in\erren,$$
i.e.
$$\lim_{t\to -\infty} u(x,t)=m\quad \forall\ x\in\erren.$$
This completes the proof of Theorem \ref{ancient}. We finish the paper with the following remark, which might be of independent interest.

\begin{remark}[{\bf{Harnack inequality}}]
It is clear from the argument above that the we have proved the following statement: there exists a positive constant $c^*$ such that, for any $T_0\in\R$, for any non-negative solution $u$ to $\cappa u=0$ in $\R^N\times (-\infty,T_0)$, and for any $z_0\in \erreu$ with $t_0<T_0$, one has
\begin{equation}\label{eccoH}
u(z)\leq (c^*)^{-1} u(z_0)\quad\mbox{ for every }z\in \Par_{z_0}\mbox{ with }t\leq t_0-1. 
\end{equation}
This can be seen as a (uniform) Harnack-type inequality under the paraboloids defined in \eqref{defpar} for non-negative ancient solutions.
\end{remark}

\bibliographystyle{amsplain}

\begin{thebibliography}{10}

\bibitem{AT}
F. Abedin, G. Tralli, \textit{Harnack inequality for a class of Kolmogorov-Fokker-Planck equations in non-divergence form}. Arch. Rational Mech. Anal. 233 (2019), 867--900

\bibitem{cup_lan_media} G. Cupini, E. Lanconelli, \textit{On mean value formulas for solutions to second order linear PDEs}. Ann. Sc. Norm. Super. Pisa Cl. Sci. (5) 22(2) (2021), 777--809 

\bibitem{GT} N. Garofalo, G. Tralli, \textit{Hardy-Littlewood-Sobolev inequalities for a class of non-symmetric and non-doubling hypoelliptic semigroups}. Math. Ann. 383 (2022), 1--38

\bibitem{Hor}
L. H{\"o}rmander,
\textit{Hypoelliptic second order differential equations}. Acta Math. 119 (1967), 147--171

\bibitem{HJ}
R.A. Horn, C.R. Johnson,
``Matrix analysis''. Cambridge University Press, Cambridge, 1990 

\bibitem{mediterranean}
A.E. Kogoj, E. Lanconelli, \textit{An invariant {H}arnack inequality for a class of hypoelliptic
  ultraparabolic equations}. Mediterr. J. Math., 1(1) (2004), 51--80

\bibitem{KLP} 
A.E. Kogoj, E. Lanconelli, E. Priola,  
\textit{Harnack inequality and Liouville-type theorems for Ornstein-Uhlenbeck and Kolmogorov operators}.
Math. Eng. 2 (2020), 690--697
\bibitem{LP} 
E. Lanconelli, S. Polidoro,  
\textit{On a class of hypoelliptic evolution operators}. 
Rend. Sem. Mat. Univ. Pol. Torino 52 (1994), 29--63
\bibitem{PZ} 
E. Priola, J. Zabczyk,  
\textit{Liouville theorems for non-local operators}. 
J. Funct. Anal. 216 (2004), 455--490
\bibitem{PZquaderni}
E. Priola, J. Zabczyk,  
\textit{Harmonic functions for generalized Mehler semigroups}, in ``Stochastic partial differential equations and applications -{VII}''. Lect. Notes Pure Appl. Math. 245 (2006), 243--256, Chapman \& Hall/CRC, Boca Raton, FL
\bibitem{P}  E. Priola, \textit{One-side Liouville theorems under an exponential growth condition for Kolmogorov operators}. 
Anal. Geom. Metr. Spaces 12 (2024), no. 1, Paper No. 20240013, 17 pp

\end{thebibliography}

\end{document}